\documentclass[pdflatex,sn-mathphys-num]{sn-jnl}

\usepackage{amsmath}
\usepackage{amssymb}
\usepackage{enumerate}
\usepackage{algorithm,algorithmic}
\usepackage{caption}
\usepackage{array}
\usepackage{pgfplots, pgfplotstable, booktabs, colortbl, array}
\usepackage{multirow}

\DeclareMathOperator{\tr}{tr}
\usepackage{float}
\floatstyle{plaintop}
\restylefloat{table}

\usepackage{todonotes}

\newcommand{\ignore}[1]{}

\usepackage{amsfonts}
\usepackage{amssymb}
\newtheorem{remark}{Remark}
\newcommand{\norm}[1]{\|#1\|}
\newcommand{\abs}[1]{|#1|}

\newcommand{\uu}{{\bf u}}


\newtheorem{theorem}{Theorem}[section]
\newtheorem{lemma}[theorem]{Lemma}

\newtheorem{assumption}[theorem]{Assumption}
\newtheorem{corollary}[theorem]{Corollary}

\theoremstyle{definition}
\newtheorem{definition}[theorem]{Definition}

\theoremstyle{remark}

\theoremstyle{thmstyleone}%

\theoremstyle{thmstyletwo}%

\theoremstyle{thmstylethree}%

\begin{document}

\title[Randomized CholeskyQR for sparse matrices]{Analysis of randomized CholeskyQR for sparse matrices}


\author*[1]{\fnm{Haoran} \sur{Guan}}\email{21037226R@connect.polyu.hk}

\author[2]{\fnm{Yuwei} \sur{Fan}}\email{fanyuwei2@huawei.com}

\affil*[1]{\orgdiv{Department of Applied Mathematics}, \orgname{The Hong Kong Polytechnic University}, \orgaddress{\street{Hung Hom}, \city{Kowloon}, \postcode{999077}, \state{Hong Kong SAR}, \country{China}}}

\affil[2]{\orgdiv{Theory Lab}, \orgname{Huawei Leibniz Research Center}, \orgaddress{\street{Sha Tin}, \city{New Territories}, \postcode{999077}, \state{Hong Kong SAR}, \country{China}}}


\abstract{This work is about rounding error analysis of randomized CholeskyQR-type algorithms for sparse matrices. We often encounter QR factorization of the sparse matrices in many real problems. In this work, we focus on some typical CholeskyQR-type algorithms with matrix sketching, which is a popular randomized technique in recent years. We build a new model of the sparse matrices and provide rounding error analysis of randomized CholeskyQR-type algorithms for the sparse cases with this model. We make comparison between the bounds with different models of sparsity both theoretically and experimentally. Numerical experiments show some new phenomena of randomized CholeskyQR-type algorithms for the sparse cases, which do not occur in the common sparse cases. We also test the applicability, accuracy, efficiency and robustness of randomized CholeskyQR-type algorithms for sparse matrices.}




\maketitle

\section{Introduction}
Let us start with the basic form of ChoelskyQR. CholeskyQR has received much attention in recent years. Among all the algorithms for QR factorization, CholeskyQR has good accuracy and efficiency compared to Householder and MGS. Moreover, it uses BLAS3 and requires significantly fewer reductions in a parallel environment than the other algorithms. CholeskyQR is shown in Algorithm~\ref{alg:cholqr}. CholeskyQR deals mainly with tall-skinny matrices. For the input $X \in \mathbb{R}^{m\times n}$ with $m \ge n$, we often have $\mbox{rank}(X)=n$. 

\begin{algorithm}[H]
\caption{$[Q,R]=\mbox{CholeskyQR}(X)$}
\label{alg:cholqr}
\begin{algorithmic}[1]
\REQUIRE $X \in \mathbb{R}^{m\times n}.$ 
\ENSURE \mbox{Orthogonal factor} $Q \in \mathbb{R}^{m\times n}$, \mbox{Upper triangular factor} $R \in \mathbb{R}^{n \times n}.$ 
\STATE $G = X^{\top}X,$
\STATE $R = \mbox{Cholesky}(G),$
\STATE $Q = XR^{-1}.$
\end{algorithmic}
\end{algorithm}%

\subsection{Development of CholeskyQR}
Although with many advantages, a single CholeskyQR is seldom used directly. With the existence of rounding errors in each step of the algorithm, CholeskyQR does not have good accuracy in orthogonality. Therefore, CholeskyQR2 \cite{2014, error} has been developed to avoid this problem by repeating CholeskyQR twice. However, CholeskyQR2 fails to deal with many ill-conditioned cases due to its Cholesky factorization and rounding errors in the first two steps. To address this issue, researchers construct Shifted CholeskyQR (SCholeskyQR) \cite{Shifted, Columns} with a shifted item in Cholesky factorization as a preconditioning step to avoid numerical breakdown when encountering ill-conditioned cases. The corresponding Shifted CholeskyQR3 (SCholeskyQR3) with CholeskyQR2 afterwards guaranties the accuracy of both orthogonality and residual. However, the existence of the shifted item influences the applicability of SCholeskyQR3. In recent years, with the fast development of randomized numerical linear algebra, \textit{e.g.}, matrix sketching, some randomized CholeskyQR-type algorithms have occurred, including \cite{Novel, Randomized, Householder, LHC}. They combine CholeskyQR with other algorithms and construct some mixed randomized algorithms to overcome the difficulty of applicability, utilizing matrix sketching to accelerate the algorithm. 

\subsection{Our considerations and contributions}
Randomized linear algebra has developed rapidly over the past several years. The idea of using randomized techniques in CholeskyQR originates in \cite{Novel}. More recently, a technique called matrix sketching has been widely used. The first two steps of CholeskyQR as shown in Algorithm~\ref{alg:cholqr} are to generate the upper-triangular $R$-factor. Then, the $Q$-factor is calculated. The primary problem of the original CholeskyQR lies in the method of receiving the $R$-factor. Among the existing algorithms for randomized CholeskyQR, some new ways to receive the $R$-factor include HouseholderQR and its connections to LU-factorization. Although with better applicability in the cases with very poor conditioning, CholeskyQR-type algorithms with a direct HouseholderQR are more efficient and can deal with most of the cases when the input matrix $X \in \mathbb{R}^{m\times n}$ satisfies $\kappa_{2}(X) \le \frac{1}{\uu}$. Therefore, to simplify the discussion, we primarily focus on CholeskyQR-type algorithms without LU factorization, including CholeskyQR2 \cite{2014, error}, SCholeskyQR3 \cite{Shifted, Columns} and Rand.Householder-Cholesky (RHC) \cite{Householder}. In fact, randomized CholeskyQR-type algorithms often use matrix sketching for the input $X$, lowering the dimension of the problem to accelerate the algorithms. This strategy can also be used in CholeskyQR2 and SCholeskyQR3 as a preconditioning step for the input $X$. Therefore, we focus on the structure of randomized CholeskyQR with matrix sketching for the input $X \in \mathbb{R}^{m\times n}$ in this work. We set $\mbox{rank}(X)=n$.

For the input matrix $X \in \mathbb{R}^{m\times n}$, if $m$ and $n$ are large, $X$ is always sparse. In many circumstances, sparse matrices exhibit different properties compared to dense matrices. Due to the structure of randomized CholeskyQR-type algorithms, we are very interested in exploring the connection between the sparse matrices and randomized CholeskyQR. Most of the existing works on CholeskyQR focus on more general cases, without considering the structure of the input $X$. In \cite{CSparse}, we discuss SCholeskyQR for the sparse cases. A different $s$ based on the number of non-zero elements and the element with the highest absolute value of $X$ is taken in order to improve applicability and accuracy. However, there is overestimation in the theoretical analysis and we do not consider the distribution of the non-zero elements and their properties. In addition, the idea of error estimation is not suitable for the randomized algorithms. Although we have to simplify the choice of the shifted item $s$ in order to guaranty the efficiency of the algorithm, we still hope to provide improved error analysis of randomized CholeskyQR-type algorithms for the sparse cases. It is also meaningful for us to explore the connection between matrix sketching and sparse matrices. 

In this work, we provide a new error analysis of randomized CholeskyQR-type algorithms for sparse matrices. To the best of our knowledge, our work is the first to explore the connection between randomized techniques and sparse matrices in rounding error analysis. We build a new model of the sparse matrices based on the existence of dense columns and the key elements of different types of columns. Using such a model, we present a new theoretical analysis of randomized CholeskyQR-type algorithms. Similar ideas of analysis can be extended to deterministic cases. We make a comparison between the error bounds using different models, both theoretically and experimentally. Numerical experiments show some distinguished phenomena of randomized CholeskyQR-type algorithms, which do not occur in the dense cases. We also test the properties of randomized CholeskyQR-type algorithms for the sparse cases, including the applicability, accuracy, efficiency, and robustness. The primary highlight of this work is the progress in theoretical analysis of randomized CholeskyQR-type algorithms for sparse matrices, providing an alternative analysis compared to the existing works.

\subsection{Outline of this work and notations}
This work is organized as follows. In the beginning, we show some theoretical results in the previous works which will be used in the theoretical analysis or comparison of this work in Section~\ref{sec:pre}. In Section~\ref{sec:ch2}, we provide a rounding error analysis of randomized CholeskyQR-type algorithms for sparse matrices. Some extensions of the ideas in this work to the deterministic CholeskyQR-type algorithms and some discussions are shown in Section~\ref{sec:extensions}. Numerical experiments are detailed in Section~\ref{sec:numerical}. Finally, we conclude with a summary of the results in Section~\ref{sec:conclusions}.

We give some notations in the following. $\norm{\cdot}_{F}$ and $\norm{\cdot}_{2}$ denote the Frobenius norm and the $2$-norm of the matrix. For $X \in \mathbb{R}^{m\times n}$ with $m \ge n$, $\norm{X}_{2}=\sigma_{1}(X)$, where $\sigma_{i}(X)$ denotes the $i$-th greatest singular value of $X$, $i=1,2,\cdots,n$. $\kappa_{2}(X)=\frac{\norm{X}_{2}}{\sigma_{n}(X)}$ is the condition number of $X$. $\uu=2^{-53}$ is the unit roundoff. For the input matrix $X$, $\abs{X}$ is the matrix whose elements are all the absolute values of the elements of $X$. $fl(\cdot)$ denotes the computed value in floating-point arithmetic. $I_{n} \in \mathbb{R}^{n\times n}$ is the identity matrix and $X^{\top}$ is the transpose of $X$. $\mbox{nnze}(\cdot)$ denotes the number of non-zero elements in the matrix.

\section{Some theoretical results from the literature}
\label{sec:pre}
Before presenting rounding error analysis of randomized CholeskyQR-type algorithms for sparse matrices, we show some theoretical results from the existing works, which will be used in the theoretical analysis of this work. 

\subsection{Lemmas for rounding error analysis}
Here, we show some lemmas regarding the rounding error analysis in \cite{Perturbation, Higham}, which are important in the analysis of this work.

\begin{lemma}[Weyl's Theorem for singular values]
\label{lemma 2.1}
If $A,B \in \mathbb{R}^{m\times n}$, then
\begin{equation}
\sigma_{min}(A+B) \ge \sigma_{min}(A)-\norm{B}_{2}. \nonumber
\end{equation}
\end{lemma}

\begin{lemma}[Rounding error in matrix multiplications]
\label{lemma 2.2}
For $A \in \mathbb{R}^{m\times n}, B \in \mathbb{R}^{n\times p}$, the error in computing the matrix product $AB$ in floating-point arithmetic is bounded by
\begin{equation}
\abs{AB-fl(AB)}\le \gamma_{n}\abs{A}\abs{B}. \nonumber
\end{equation}
Here, $\abs{A}$ is the matrix whose $(i,j)$ element is $\abs{a_{ij}}$ and
\begin{equation}
\gamma_n: = \frac{n{\uu}}{1-n{\uu}} \le 1.02n{\uu}. \nonumber
\end{equation}
\end{lemma}

\begin{lemma}[Rounding error in Cholesky factorization]
\label{lemma 2.3}
If Cholesky factorization applied to the symmetric positive definite $A \in \mathbb{R}^{n\times n}$ runs to completion, then the computed factor $R \in \mathbb{R}^{n\times n}$ satisfies
\begin{equation}
R^{\top}R=A+\Delta{A}, \quad \abs{\Delta A}\le \gamma_{n+1}\abs{{R}^{\top}}\abs{R}. \nonumber
\end{equation}
\end{lemma}

\subsection{Matrix sketching}
\label{sec:ms}
In this section, we show some existing results regarding matrix sketching. We start with the definition of an $\epsilon$-subspace embedding \cite{rgs, Fast, 4031351}. 

\begin{definition}[$\epsilon$-subspace embedding]
\label{definition 1}
When there is an $\epsilon$ which satisfies $0 \le \epsilon <1$, the sketch matrix $\Omega \in \mathbb{R}^{s\times n}$ is an $\epsilon$-subspace embedding for the subspace $\mathcal{K} \subset \mathbb{R}^{n}$, if for any $x,y \in \mathcal{K}$, 
\begin{equation}
\abs{\langle x,y \rangle-\langle \Omega x,\Omega y \rangle} \le \epsilon\norm{x}_{2}\norm{y}_{2}. \nonumber
\end{equation}
Here, $\langle \cdot,\cdot \rangle$ is the Euclidean inner product for the vectors. 
\end{definition}

$\epsilon$-subspace embeddings, which require knowledge of the subspace $\mathcal{K}$. In contrast, a concept known as the $(\epsilon,p,n)$ oblivious $l_{2}$-subspace embedding \cite{rgs} does not require this knowledge.

\begin{definition}[$(\epsilon,p,n)$ oblivious $l_{2}$-subspace embedding]
\label{definition 2}
When $\Omega \in \mathbb{R}^{s\times m}$ is an $\epsilon$-subspace embedding for all the fixed $n$-dimensional subspace $\mathcal{K} \subset \mathbb{R}^{m}$ with probability at least $1-p$, it is an $(\epsilon,p,n)$ oblivious $l_{2}$-subspace embedding.
\end{definition}

In the following, we list some properties regarding matrix-sketching based on Definition~\ref{definition 2}.

\begin{lemma}[Some properties of matrix sketching]
\label{lemma 22}
If $\Omega \in \mathbb{R}^{s\times m}$ is a $(\epsilon,p,n)$ oblivious $l_{2}$-subspace embedding in $\mathbb{R}^{m}$, then for any $n$-dimensional subspace $\mathcal{K} \subset \mathbb{R}^{m}$ and $X \in \mathbb{R}^{m\times n}$, we have
\begin{align}
\sqrt{1-\epsilon} \cdot \norm{X}_{2} &\le \norm{\Omega X}_{2} \le \sqrt{1+\epsilon} \cdot \norm{X}_{2}, \label{eq:222} \\
\sqrt{1-\epsilon} \cdot \norm{X}_{F} &\le \norm{\Omega X}_{F} \le \sqrt{1+\epsilon} \cdot \norm{X}_{F}, \label{eq:2f} \\
\sqrt{1-\epsilon} \cdot \sigma_{min}(X) &\le \sigma_{min}(\Omega X) \le \norm{\Omega X}_{2} \le \sqrt{1+\epsilon} \cdot \norm{X}_{2}, \label{eq:2e} \\
\frac{\sigma_{min}(\Omega X)}{\sqrt{1+\epsilon}} &\le \sigma_{min}(X) \le \norm{X}_{2} \le \frac{\norm{\Omega X}_{2}}{\sqrt{1-\epsilon}}, \label{eq:2g} 
\end{align}
with probability at least $1-p$.
\end{lemma}

There are several different methods of matrix sketching, including the Gaussian sketch, the CountSketch and SRHT, \textit{etc}. For more about matrix sketching, readers can refer to \cite{rgs, pmlr, Fast, Estimating, pylspack} and their references. Among all the methods of matrix sketching, the Gaussian sketch is the most common one. The sketching matrix of the Gaussian sketch is defined as $\Omega=\frac{1}{s}G$, where $G \in \mathbb{R}^{s\times m}$ is a Gaussian matrix. The sketching size $s$ \cite{4031351} satisfies
\begin{equation}
s=z \cdot \frac{log n \cdot log{\frac{1}{p}}}{\epsilon^{2}}. \label{eq:e1}
\end{equation}
Here, $z$ is a positive constant. In many real cases, we can take $s=yn$ for the Gaussian sketch, where $y \ge 1$. In this work, we focus on randomized CholeskyQR-type algorithms with the Gaussian sketch. 

\subsection{Theoretical results of some existing CholeskyQR-type algorithms}
In this part, we present theoretical results of some representative existing CholeskyQR-type algorithms, including CholeskyQR2 \cite{error} and RHC \cite{Householder}. CholeskyQR2 and RHC are shown below in Algorithm~\ref{alg:cholqr2} and Algorithm~\ref{alg:Rand2}. RHC in \cite{Householder} utilizes a new type of matrix sketching, multi-sketching. We present the version with single-sketching here.

\begin{algorithm}
\caption{$[Q,R]=\mbox{CholeskyQR2}(X)$}
\label{alg:cholqr2}
\begin{algorithmic}[1]
\REQUIRE $X \in \mathbb{R}^{m\times n}$ with $\mbox{rank}(X)=n$.
\ENSURE \mbox{Orthogonal factor} $Q \in \mathbb{R}^{m\times n}$, \mbox{Upper triangular factor} $R \in \mathbb{R}^{n \times n}.$ 
\STATE $[W,Y]=\mbox{CholeskyQR}(X),$
\STATE $[Q,Z]=\mbox{CholeskyQR}(Q),$
\STATE $R=ZY.$
\end{algorithmic}
\end{algorithm}%

\begin{algorithm}
\caption{$[Q,R]=\mbox{RHC}(X)$}
\label{alg:Rand2}
\begin{algorithmic}[1]
\REQUIRE $X \in \mathbb{R}^{m\times n}$ with $\mbox{rank}(X)=n$.
\ENSURE \mbox{Orthogonal factor} $Q \in \mathbb{R}^{m\times n}$, \mbox{Upper triangular factor} $R \in \mathbb{R}^{n \times n}.$ 
\STATE $K=\Omega X,$
\STATE $[H,Y]=\mbox{HouseholderQR}(K),$
\STATE $W=XY^{-1}.$
\STATE $[Q,Z]=\mbox{CholeskyQR}(W),$
\STATE $R=ZY.$
\end{algorithmic}
\end{algorithm}%

In the following, we present the rounding error analysis of CholeskyQR2 and RHC in Lemma~\ref{lemma 2.9} and Lemma~\ref{lemma 2.10}. 

\begin{lemma}[Rounding error analysis of CholeskyQR2]
\label{lemma 2.9}
For $X \in \mathbb{R}^{m\times n}$ and $[Q,R]=\mbox{CholeskyQR2}(X)$, with
$\delta_{0}=8\kappa_{2}(X)\sqrt{mn\uu+n(n+1)\uu} \le 1$, $mn\uu \le \frac{1}{64}$ and $n(n+1)\uu \le \frac{1}{64}$, we have
\begin{align}
\norm{Q^{\top}Q-I}_{F} &\le 6(mn\uu+n(n+1)\uu), \label{eq:13} \\
\norm{QR-X}_{F} &\le 5n^{2}\sqrt{n}\uu\norm{X}_{2}. \label{eq:14}
\end{align}
\end{lemma}

\begin{lemma}[Rounding error analysis of RHC]
\label{lemma 2.10}
Suppose $0 \le \epsilon<\frac{308}{317}$ and $S \in \mathbb{R}^{p\times n}$ is a $(\epsilon,p,n)$ oblivious $l_{2}$-subspace embedding. Furthermore, suppose $X \in \mathbb{R}^{m\times n}$ has full rank and $1<n \le s \le m$ where $mn\uu \le \frac{1}{12}$, $s^{\frac{3}{2}}\uu \le \frac{1}{12}$ and $\delta=\frac{383(sn^{\frac{3}{2}}+\sqrt{n}(p^{\frac{3}{2}}\sqrt{1+\epsilon}+n\norm{S}_{F}}{\sqrt{1-\epsilon}}\uu\kappa_{2}(X) \le 1$. When the above assumptions are satisfied, for $[Q,R]=\mbox{RHC}(X)$, we have
\begin{align}
\norm{Q^{\top}Q-I}_{F} &\le \frac{5445}{(25\sqrt{\frac{1-\epsilon}{1+\epsilon}}-3)^{2}}(mn\uu+n(n+1)\uu), \label{eq:15} \\
\norm{QR-X}_{F} &\le (\frac{56}{25\frac{1-\epsilon}{\sqrt{1+\epsilon}}-3\sqrt{1-\epsilon}}+\frac{1.5}{\sqrt{1-\epsilon}}\sqrt{1+\frac{5445(mn\uu+n(n+1)\uu}{(25\sqrt{\frac{1-\epsilon}{1+\epsilon}}-3)^{2}}}) \nonumber \\ &\cdot (\sqrt{1+\epsilon}\norm{X}_{2}+\frac{1-\epsilon}{12}\sigma_{n}(X)\delta)n^{2}\uu+\frac{\delta}{10}\sigma_{n}(X), \label{eq:16}
\end{align}
with probability at least $1-d$.
\end{lemma}

\section{Rounding error analysis of randomized CholeskyQR for sparse matrices}
\label{sec:ch2}
In this section, we focus on randomized CholeskyQR-type algorithms. We provide the general structure of randomized CholeskyQR-type algorithms and a new model based on the structure and elements of the sparse matrices. Our analysis is based on this model.

\subsection{The general structure of randomized CholeskyQR}
As a popular technique, which is widely used in many problems of numerical linear algebra, matrix sketching can accelerate the algorithm while maintaining accuracy with high probability. The primary problem of CholeskyQR lies in the steps to generate the upper-triangular factor. In \cite{Householder}, a step of matrix sketching is in front of HouseholderQR to accelerate the algorithms. The steps of calculating the gram matrix in CholeskyQR2 and SCholeskyQR3 for tall-skinny matrices restrict the applicability of the algorithms. With the advantage of decreasing the dimension of the problem, it is applicable for us to put a step of matrix sketching before implementing the whole algorithm to improve the applicability of the algorithm. We take CholeskyQR2 as an example. We construct randomized CholeskyQR2 (RCholeskyQR2) with matrix sketching as shown in Algorithm~\ref{alg:MR}. Here, $\Omega\in \mathbb{R}^{s\times m}$ is the matrix of the Gaussian Sketch with $n \le s \le m$.

\begin{algorithm}
\caption{$[Q,R]=\mbox{RCholeskyQR2}(X)$}
\label{alg:MR}
\begin{algorithmic}
\REQUIRE $X \in \mathbb{R}^{m\times n}.$ 
\ENSURE \mbox{Orthogonal factor} $Q \in \mathbb{R}^{m\times n}$, \mbox{Upper triangular factor} $R \in \mathbb{R}^{n \times n}.$
\STATE $A=\Omega X,$
\STATE $G=A^{\top}A,$
\STATE $Y=\mbox{Cholesky}(G),$
\STATE $W=XY^{-1}.$
\STATE $[Q,Z]=\mbox{CholeskyQR}(W),$
\STATE $R=ZY.$
\end{algorithmic}
\end{algorithm}

Combining RCholeskyQR2 with RHC, we can find that randomized CholeskyQR can be summarized in the form of 'matrix sketching+algorithm'. This can be viewed as a general structure for this type of algorithms. SCholeskyQR3 with matrix sketching can be constructed in the same manner.

\subsection{A new model of the sparse matrices}
In this part, we provide a new model of sparse matrices. In \cite{CSparse}, a new shifted item $s$ based on the number of non-zero elements and the element with the largest absolute value is taken for SCholeskyQR3 for the sparse input $X$. Although such an idea guaranties the efficiency of the algorithm, it is still hard to describe the distribution of the non-zero elements and their absolute values for the sparse matrices. From the perspective of theoretical analysis in this work, we propose a new model of the sparse matrices based on the division of columns and some definitions of different types of non-zero elements. It is shown in Definition~\ref{def:41}. 

\begin{definition}[The new model of the sparse matrices based on the division of columns]
\label{def:41}
A sparse matrix $X \in \mathbb{R}^{m\times n}$ has $v$ dense columns ($0 \le v<<n$, with each dense column containing at most $t_{1}$ non-zero elements. For the sparse columns of $X$, each column has at most $t_{2}$ non-zero elements, where $0 \le t_{2}<<t_{1}$. The largest absolute value of the elements in the dense columns is taken as $c_{1}$ and the largest absolute value of the elements in the sparse columns is taken as $c_{2}$. 
\end{definition}

\subsection{RCholeskyQR2 for sparse matrices}
Here, we do rounding error analysis of randomized CholeskyQR for sparse matrices based on the model in Definition~\ref{def:41}. In this work, we focus mainly on RCholeskyQR2 in Algorithm~\ref{alg:MR}. SCholeskyQR3 and RHC for sparse matrices based on our new model can be analyzed similarly.

\subsubsection{Settings of RCholeskyQR2 for sparse matrices}
In the beginning, we write RCholeskyQR2 with error matrices in each step below.

\begin{align}
A-\Omega X &= E_{S}, \label{eq:es} \\
G-A^{\top}A &= E_{G}, \label{eq:35} \\
Y^{\top}Y-G &= E_{C}, \label{eq:36} \\
WY-X &= E_{X}, \label{eq:37} \\
C-W^{\top}W &= E_{1}, \nonumber \\
Z^{\top}Z-C &= E_{2}, \nonumber \\
QZ-W &= E_{3}, \label{eq:310} \\
ZY-R &= E_{4}. \label{eq:311}
\end{align}

For RCholeskyQR2 with matrix sketching, we focus on $(\epsilon, p,n)$ oblivious $l_{2}$-subspace embedding in Definition~\ref{definition 2}. We provide the following assumption. 

\begin{assumption}[The assumption for matrix sketching of RCholeskyQR2]
\label{assumption:2}
If $\Omega\in \mathbb{R}^{s\times m}$ is a $(\epsilon,p,n)$ oblivious $l_{2}$-subspace embedding in $\mathbb{R}^{m}$, we let
\begin{equation}
\sqrt{\frac{1-\epsilon}{1+\epsilon}} \ge \frac{116}{11}\sqrt{mn\uu+n(n+1)\uu}+\frac{3}{88}. \nonumber
\end{equation}
\end{assumption}

Moreover, when $X \in \mathbb{R}^{m\times n}$ is a sparse matrix, we set
\begin{equation}
\kappa_{2}(X) \le \delta_{1}. \label{eq:rw}
\end{equation}
Here, we define $\delta_{1}=\min(\frac{\sqrt{1-\epsilon}}{10.5k}, \frac{1}{10.5j \cdot (\sqrt{1+\epsilon}+1.02m\uu \cdot \norm{\Omega}_{F})} \cdot \sqrt{\frac{1-\epsilon}{s\uu+(n+1)\uu}})$, where $j=\frac{\sqrt{vt_{1}c_{1}^{2}+nt_{2}c_{2}^{2}}}{\norm{X}_{2}}$ and $k=1.02m\uu \cdot \norm{\Omega}_{F} \cdot j$. We have $n \le s \le m$. $c$, $v$, $t_{1}$ and $t_{2}$ are defined in Definition~\ref{def:41}. 

Also, for $X \in \mathbb{R}^{m\times n}$, we let
\begin{align}
mn\uu &\le \frac{1}{64}, \label{eq:31} \\
n(n+1)\uu &\le \frac{1}{64}. \label{eq:32}
\end{align}

\subsubsection{Theoretical results of RCholeskyQR2 for sparse matrices}
After providing the general settings and assumptions above, we show the theoretical results of RCholeskyQR2 for sparse matrices in Theorem~\ref{theorem 3.3}. Similar theoretical results hold for RHC.

\begin{theorem}[Rounding error analysis of RCholeskyQR2 for sparse matrices]
\label{theorem 3.3}
When $X \in \mathbb{R}^{m\times n}$ is a sparse matrix with $\mbox{rank}(X)=n$ and $m \ge n$, if Assumption~\ref{assumption:2} and \eqref{eq:rw}-\eqref{eq:32} are satisfied, then for $[Q,R]=\mbox{RCholeskyQR2}(X)$, we have
\begin{align}
\norm{Q^{\top}Q-I_{n}}_{F} &\le d, \label{eq:372} \\
\norm{QR-X}_{F} &\le b\norm{X}_{2}, \label{eq:382}
\end{align}
with probability at least $1-p$. Here, we define $d=5a^{2} \cdot (mn\uu+n(n+1)\uu$, $a=\frac{1.16}{0.88 \cdot \sqrt{\frac{1-\epsilon}{1+\epsilon}}-0.03}$ and $b=\frac{(1.22+1.24\sqrt{1+d})}{\sqrt{1-\epsilon}} \cdot (\sqrt{1+\epsilon}+1.02m\uu \cdot \norm{\Omega}_{F}) \cdot jn\sqrt{n}\uu+ \frac{1.24\sqrt{1+d}}{\sqrt{1-\epsilon}} \cdot (\sqrt{1+\epsilon}+k) \cdot n^{2}\uu$. $k$ and $j$ are defined in \eqref{eq:rw}. 
\end{theorem}

Theorem~\ref{theorem 3.3} demonstrates the accuracy of RCholeskyQR2 for sparse matrices. The most important task of this work is to prove it.

\subsection{Lemmas for proving Theorem~\ref{theorem 3.3}}
Before proving Theorem~\ref{theorem 3.3}, we show some lemmas regarding RCholeskyQR2 based on the assumptions and settings above. Theoretical analysis of RCholeskyQR2 is distinguished from that of CholeskyQR2 in \cite{error} due to the existence of randomized matrix sketching and our new model of sparse matrices. We follow Definition~\ref{def:41} and present an improved analysis of CholeskyQR-type algorithms for sparse matrices compared to that in \cite{CSparse}, utilizing some ideas in \cite{Householder, LHC}. 

\begin{lemma}[Estimating $\norm{E_{S}}_{2}$]
\label{lemma 3.1}
For $E_{S}$ in \eqref{eq:es}, with the settings in Definition~\ref{def:41}, we have
\begin{equation}
\begin{split}
\norm{E_{S}}_{F} &\le 1.02m\uu \cdot \norm{\Omega}_{F} \cdot \sqrt{vt_{1}c_{1}^{2}+nt_{2}c_{2}^{2}} \\ &= k\norm{X}_{2}, \label{eq:esk}
\end{split}
\end{equation}
with probability at least $1-p$. Here, $k=\frac{1.02m\uu \cdot \norm{\Omega}_{F} \cdot \sqrt{vt_{1}c_{1}^{2}+nt_{2}c_{2}^{2}}}{\norm{X}_{2}}$.
\end{lemma}
\begin{proof}
To estimate $\norm{E_{S}}_{F}$, based on Lemma~\ref{lemma 2.2}, Lemma~\ref{lemma 22} and Definition~\ref{def:41}, we can bound $\norm{E_{S}}_{F}$ as 
\begin{equation} \label{eq:es1f}
\begin{split}
\norm{E_{S}}_{F} &\le 1.02m\uu \cdot \norm{\Omega}_{F}\norm{X}_{F} \\ &\le 1.02m\uu \cdot \norm{\Omega}_{F} \cdot \sqrt{vt_{1}c_{1}^{2}+nt_{2}c_{2}^{2}}, 
\end{split}
\end{equation}
with probability at least $1-p$. Therefore, it is clear that $k=\frac{1.02m\uu \cdot \norm{\Omega}_{F} \cdot \sqrt{vt_{1}c_{1}^{2}+nt_{2}c_{2}^{2}}}{\norm{X}_{2}}$. \eqref{eq:esk} is proved.
\end{proof}

\begin{lemma}[Estimating $\norm{E_{G}}_{2}$ and $\norm{E_{C}}_{2}$]
\label{lemma 3.2}
For $E_{G}$ and $E_{C}$ in \eqref{eq:35} and \eqref{eq:36}, we have
\begin{align}
\norm{E_{G}}_{F} &\le 1.1s\uu \cdot (\sqrt{1+\epsilon}+1.02m\uu \cdot \norm{\Omega}_{F})^{2} \cdot (vt_{1}c_{1}^{2}+nt_{2}c_{2}^{2}), \label{eq:ea} \\
\norm{E_{C}}_{F} &\le 1.1(n+1)\uu \cdot (\sqrt{1+\epsilon}+1.02m\uu \cdot \norm{\Omega}_{F})^{2} \cdot (vt_{1}c_{1}^{2}+nt_{2}c_{2}^{2}), \label{eq:eb}
\end{align}
with probability at least $1-p$.
\end{lemma}
\begin{proof}
For $\norm{E_{G}}_{F}$, based on Lemma~\ref{lemma 2.2}, we can bound $\norm{E_{G}}_{F}$ as
\begin{equation}
\norm{E_{G}}_{F} \le 1.1s\uu \cdot \norm{A}_{F}^{2}. \label{eq:egf}
\end{equation}
With \eqref{eq:2f}, \eqref{eq:es}, Definition~\ref{def:41} and \eqref{eq:esk}, we can estimate $\norm{A}_{F}$ as
\begin{equation} \label{eq:af}
\begin{split}
\norm{A}_{F} &\le \norm{\Omega X}_{F}+\norm{E_{S}}_{F} \\ &\le \sqrt{1+\epsilon} \cdot \sqrt{vt_{1}c_{1}^{2}+nt_{2}c_{2}^{2}}+1.02m\uu \cdot \norm{\Omega}_{F} \cdot \sqrt{vt_{1}c_{1}^{2}+nt_{2}c_{2}^{2}} \\ &= (\sqrt{1+\epsilon}+1.02m\uu \cdot \norm{\Omega}_{F}) \cdot \sqrt{vt_{1}c_{1}^{2}+nt_{2}c_{2}^{2}},
\end{split}
\end{equation}
with probability at least $1-p$. Therefore, we put \eqref{eq:af} into \eqref{eq:egf} and we can get \eqref{eq:ea}. \eqref{eq:ea} is proved.

For $\norm{E_{C}}_{F}$, with Lemma~\ref{lemma 2.3}, \eqref{eq:35}, \eqref{eq:36} and the connection between $\norm{\cdot}_{F}$ and the trace of the matrix, we can have
\begin{equation} \label{eq:ecf}
\begin{split}
\norm{E_{C}}_{F} &\le 1.02(n+1)\uu \cdot \norm{Y}_{F}^{2} \\ &= 1.02(n+1)\uu \cdot \tr(Y^{\top}Y) \\ &= 1.02(n+1)\uu \cdot \tr(A^{\top}A+E_{G}+E_{C}) \\ &\le 1.02(n+1)\uu \cdot (\norm{A}_{F}^{2}+n\norm{E_{G}}_{F}+n\norm{E_{C}}_{F}. 
\end{split}
\end{equation}
Therefore, with \eqref{eq:31}, \eqref{eq:32}, \eqref{eq:ea}, \eqref{eq:af} and \eqref{eq:ecf}, we can bound $\norm{E_{C}}_{F}$ as
\begin{equation}
\begin{split}
\norm{E_{C}}_{F} &\le \frac{1.02(n+1)\uu \cdot (\norm{A}_{F}^{2}+n\norm{E_{G}}_{F})}{1-1.02(n+1)\uu \cdot n} \\ &\le \frac{1.02(n+1)\uu \cdot (1+1.1s_{2}\uu \cdot n)}{1-1.02(n+1)\uu \cdot n} \cdot (\sqrt{1+\epsilon}+1.1m\uu \cdot \norm{\Omega}_{F})^{2} \cdot (vt_{1}c_{1}^{2}+nt_{2}c_{2}^{2}) \\ &\le \frac{1.02(n+1)\uu \cdot (1+\frac{1}{64})}{1-\frac{1}{64}} \cdot (\sqrt{1+\epsilon}+1.1m\uu \cdot \norm{\Omega}_{F})^{2} \cdot (vt_{1}c_{1}^{2}+nt_{2}c_{2}^{2}) \\ &\le 1.1(n+1)\uu \cdot (\sqrt{1+\epsilon}+1.1m\uu \cdot \norm{\Omega}_{F})^{2} \cdot (vt_{1}c_{1}^{2}+nt_{2}c_{2}^{2}), \nonumber
\end{split}
\end{equation}
with probability at least $1-p$. \eqref{eq:eb} is proved.
\end{proof}

\begin{lemma}[Estimating $\norm{Y^{-1}}_{2}$ and $\norm{AY^{-1}}_{2}$]
\label{lemma 3.3}
For $Y^{-1}$ and $AY^{-1}$ in \eqref{eq:36}, we have
\begin{align}
\norm{Y^{-1}}_{2} &\le \frac{1.12}{\sqrt{1-\epsilon} \cdot \sigma_{n}(X)}, \label{eq:322} \\
\norm{AY^{-1}}_{2} &\le 1.01, \label{eq:323}
\end{align}
with probability at least $1-p$.
\end{lemma}
\begin{proof}
For $\norm{Y^{-1}}_{2}$, with Lemma~\ref{lemma 2.1}, \eqref{eq:35} and \eqref{eq:36}, we can get
\begin{equation}
(\sigma_{n}(Y))^{2} \ge (\sigma_{n}(A))^{2}-(\norm{E_{G}}_{2}+\norm{E_{C}}_{2}). \label{eq:ny2}
\end{equation}
For \eqref{eq:es}, with Lemma~\ref{lemma 2.1}, \eqref{eq:2e} and \eqref{eq:esk}, we can get
\begin{equation} \label{eq:signa}
\begin{split}
\sigma_{n}(A) &\ge \sigma_{n}(\Omega X)-\norm{E_{S}}_{2} \\ &\ge \sqrt{1-\epsilon} \cdot \sigma_{n}(X)-k\norm{X}_{2} \\ &\ge 0.9\sqrt{1-\epsilon} \cdot \sigma_{n}(X),
\end{split}
\end{equation}
with probability at least $1-p$. We put \eqref{eq:signa} into \eqref{eq:ny2}, together with \eqref{eq:rw}, \eqref{eq:ea} and \eqref{eq:eb}, we can get
\begin{equation} \label{eq:signa1}
\begin{split}
(\sigma_{n}(Y))^{2} &\ge (\sigma_{n}(A))^{2}-(\norm{E_{G}}_{2}+\norm{E_{C}}_{2}) \\ &\ge 0.81(1-\epsilon) \cdot (\sigma_{n}(X))^{2}-1.1s\uu \cdot (\sqrt{1+\epsilon}+1.02m\uu \cdot \norm{\Omega}_{F})^{2} \cdot (vt_{1}c_{1}^{2}+nt_{2}c_{2}^{2}) \\ &-1.1(n+1)\uu \cdot (\sqrt{1+\epsilon}+1.02m\uu \cdot \norm{\Omega}_{F})^{2} \cdot (vt_{1}c_{1}^{2}+nt_{2}c_{2}^{2}) \\ &\ge 0.8(1-\epsilon) \cdot (\sigma_{n}(X))^{2},
\end{split}
\end{equation}
with probability at least $1-p$. Therefore, with \eqref{eq:signa1}, we can bound $\norm{Y^{-1}}$ as 
\begin{equation} 
\begin{split}
\norm{Y^{-1}}_{2} &= \frac{1}{\sigma_{n}(Y)} \nonumber \\ &\le \frac{1}{\sqrt{0.8(1-\epsilon) \cdot (\sigma_{n}(X))^{2}}} \nonumber \\ &\le \frac{1.12}{\sqrt{1-\epsilon} \cdot \sigma_{n}(X)}, \nonumber 
\end{split}
\end{equation}
with probability at least $1-p$. \eqref{eq:322} is proved. 

For $\norm{AY^{-1}}_{2}$, with \eqref{eq:35} and \eqref{eq:36}, we can get
\begin{equation}
Y^{-\top}A^{\top}AY^{-1}=I_{n}-Y^{-\top}(E_{C}+E_{G})Y^{\top}. \label{eq:ya2}
\end{equation}
Based on \eqref{eq:ya2}, we can get
\begin{equation}
\norm{AY^{-1}}_{2}^{2} \le 1+\norm{Y^{-1}}_{2}^{2}(\norm{E_{C}}_{2}+\norm{E_{G}}_{2}). \label{eq:ay-12}
\end{equation}
With \eqref{eq:rw}, \eqref{eq:ea} and \eqref{eq:eb}, we can get
\begin{equation}
\norm{E_{C}}_{2}+\norm{E_{G}}_{2} \le 0.01(1-\epsilon) \cdot (\sigma_{n}(X))^{2}, \label{eq:cg2}
\end{equation}
with probability at least $1-p$. We put \eqref{eq:322} and \eqref{eq:cg2} into \eqref{eq:ay-12} and we can get
\begin{equation} 
\begin{split}
\norm{AY^{-1}}_{2} &\le \sqrt{1+\norm{Y^{-1}}_{2}^{2}(\norm{E_{C}}_{2}+\norm{E_{G}}_{2})} \\ &\le \sqrt{1+\frac{1.25}{(1-\epsilon) \cdot (\sigma_{n}(X))^{2}} \cdot 0.01(1-\epsilon) \cdot (\sigma_{n}(X))^{2}} \nonumber \\ &\le 1.01, \nonumber
\end{split}
\end{equation}
with probability at least $1-p$. \eqref{eq:323} is proved.
\end{proof}

\begin{lemma}[Estimating $\norm{E_{X}}_{2}$]
\label{lemma 3.4}
For $E_{X}$ in \eqref{eq:37}, we have
\begin{align}
\norm{E_{X}}_{F} \le 1.05n\sqrt{n}\uu \cdot \norm{W}_{2} \cdot (\sqrt{1+\epsilon}+1.02m\uu \cdot \norm{\Omega}_{F}) \cdot \sqrt{vt_{1}c_{1}^{2}+nt_{2}c_{2}^{2}}, \label{eq:327}
\end{align}
with probability at least $1-p$.
\end{lemma}
\begin{proof}
Based on Lemma~\ref{lemma 2.2} and \eqref{eq:37}, we can get
\begin{equation}
\norm{E_{X}}_{F} \le 1.02n\uu \cdot \norm{W}_{F}\norm{Y}_{F}. \label{eq:exf}
\end{equation}
With \eqref{eq:31}, \eqref{eq:32}, \eqref{eq:ea}, \eqref{eq:eb}, \eqref{eq:af} and \eqref{eq:ecf}, we can have
\begin{equation} \label{eq:nyf2}
\begin{split}
\norm{Y}_{F}^{2} &\le \norm{A}_{F}^{2}+n\norm{E_{C}}_{F}+n\norm{E_{G}}_{F} \\ &\le (\sqrt{1+\epsilon}+1.02m\uu \cdot \norm{\Omega}_{F})^{2} \cdot (vt_{1}c_{1}^{2}+nt_{2}c_{2}^{2}) \\ &+n \cdot 1.1s\uu \cdot (\sqrt{1+\epsilon}+1.02m\uu \cdot \norm{\Omega}_{F})^{2} \cdot (vt_{1}c_{1}^{2}+nt_{2}c_{2}^{2}) \\ &+ n \cdot 1.1(n+1)\uu \cdot (\sqrt{1+\epsilon}+1.02m\uu \cdot \norm{\Omega}_{F})^{2} \cdot (vt_{1}c_{1}^{2}+nt_{2}c_{2}^{2}) \\ &\le 1.04 \cdot (\sqrt{1+\epsilon}+1.02m\uu \cdot \norm{\Omega}_{F})^{2} \cdot (vt_{1}c_{1}^{2}+nt_{2}c_{2}^{2}),
\end{split}
\end{equation}
with probability at least $1-p$. With \eqref{eq:nyf2}, we can bound $\norm{Y}_{F}$ as
\begin{equation} \label{eq:nyf}
\begin{split}
\norm{Y}_{F} \le 1.02 \cdot (\sqrt{1+\epsilon}+1.02m\uu \cdot \norm{\Omega}_{F}) \cdot \sqrt{vt_{1}c_{1}^{2}+nt_{2}c_{2}^{2}},
\end{split}
\end{equation}
with probability at least $1-p$. We put \eqref{eq:nyf} into \eqref{eq:exf} and we can get
\begin{equation}
\begin{split}
\norm{E_{X}}_{F} &\le 1.02n\uu \cdot \norm{W}_{F}\norm{Y}_{F} \nonumber \\ &\le 1.02n\sqrt{n}\uu \cdot \norm{W}_{2}\norm{Y}_{F} \nonumber \\ &\le 1.05n\sqrt{n}\uu \cdot \norm{W}_{2} \cdot (\sqrt{1+\epsilon}+1.02m\uu \cdot \norm{\Omega}_{F}) \cdot \sqrt{vt_{1}c_{1}^{2}+nt_{2}c_{2}^{2}}, \nonumber
\end{split}
\end{equation}
with probability at least $1-p$. \eqref{eq:exf} is proved.
\end{proof}

\begin{lemma}[Estimating $\sigma_{n}(AY^{-1})$]
\label{lemma 3.5}
For $\sigma_{n}(AY^{-1})$ in \eqref{eq:36}, we have
\begin{equation}
\sigma_{n}(AY^{-1}) \ge 0.99, \label{eq:nay-1}
\end{equation}
with probability at least $1-p$.
\end{lemma}
\begin{proof}
To estimate $\sigma_{n}(AY^{-1})$, we take consideration of $A$ first. We do SVD for $A$ in the form of $A=U\Sigma V^{\top}$. Here, $U \in \mathbb{R}^{s\times s}$ and $V \in \mathbb{R}^{n\times n}$ are orthogonal matrices. $\Sigma \in \mathbb{R}^{s\times n}$ is the matrix with singular values of $A$. Therefore, we can have
\begin{equation}
\begin{split}
A^{\top}A &= (U\Sigma V^{\top})^{\top}(U\Sigma V^{\top}) \\ &= V\Sigma^{\top}\Sigma V^{\top}. \label{eq:ata}
\end{split}
\end{equation}
We set $D=AY^{-1}=U\Sigma V^{\top}Y^{-1}$. With \eqref{eq:35} and \eqref{eq:36}, it is easy for us to have
\begin{equation} \label{eq:dtd}
\begin{split}
D^{\top}D &= Y^{-\top}(A^{\top}A)Y^{-1} \\ &=Y^{-\top}(Y^{\top}Y+E_{G}+E_{C})Y^{-1} \\ &= I_{n}-Y^{-\top}(E_{G}+E_{C})Y^{-1}. 
\end{split}
\end{equation}
Regarding $Y^{-\top}(E_{G}+E_{C})Y^{-1}$, with \eqref{eq:322} and \eqref{eq:cg2}, we can get
\begin{equation} \label{eq:0.014}
\begin{split}
\norm{Y^{-\top}(E_{G}+E_{C})Y^{-1}}_{2} &\le \norm{Y^{-1}}_{2}^{2}(\norm{E_{G}}_{2}+\norm{E_{C}}_{2}) \\ &\le (\frac{1.12}{\sqrt{1-\epsilon} \cdot \sigma_{n}(X)})^{2} \cdot 0.01 \cdot (1-\epsilon) \cdot (\sigma_{n}(X))^{2} \\ &\le 0.02, 
\end{split}
\end{equation}
with probability at least $1-p$. With \eqref{eq:dtd} and \eqref{eq:0.014}, we can have
\begin{equation} \label{eq:ddi}
\begin{split}
\norm{D^{\top}D-I_{n}}_{2} &= \norm{Y^{-\top}(E_{G}+E_{C})Y^{-1}}_{2} \\ &\le 0.02,
\end{split}
\end{equation}
with probability at least $1-p$. With \eqref{eq:ddi}, we can get
\begin{equation}
\norm{D}_{2} \le 1.01, \label{eq:d2} 
\end{equation}
with probability at least $1-p$. If we do SVD for $D$ in the form of $D=MBN^{\top}$ with $B=L+E$, where $L \in \mathbb{R}^{s\times n}$ is the diagonal matrix with all the elements on the diagonal $1$, with \eqref{eq:d2}, we can bound $\norm{E}_{2}$ as
\begin{equation}
\norm{E}_{2} \le 0.01, \label{eq:ne2}
\end{equation}
with probability at least $1-p$. We combine SVD of $D$ with $D=AY^{-1}=U\Sigma V^{\top}Y^{-1}$ and we can get
\begin{equation}
MBN^{\top}=U\Sigma V^{\top}Y^{-1}. \label{eq:equa}
\end{equation}
Since $U$ is orthogonal, it is easy to see that $\sigma_{n}(AY^{-1})=\sigma_{n}(D)=\sigma_{n}(\Sigma V^{\top}Y^{-1})$. With \eqref{eq:equa}, we can get
\begin{equation}
\Sigma V^{\top}Y^{-1}=U^{\top}MBN^{\top}. \label{eq:vy-1}
\end{equation}
When $U$, $M$ and $N$ are all orthogonal matrices, with Lemma~\ref{lemma 2.1}, \eqref{eq:ne2} and \eqref{eq:vy-1}, we can get
\begin{equation} 
\begin{split}
\sigma_{n}(AY^{-1}) &= \sigma_{n}(\Sigma V^{\top}Y^{-1}) \nonumber \\ &= \sigma_{n}(B) \nonumber \\ &\ge \sigma_{n}(L)-\norm{E}_{2} \nonumber \\ &\ge 1-0.01 \nonumber \\ &\ge 0.99, \nonumber
\end{split}
\end{equation}
with probability at least $1-p$. \eqref{eq:nay-1} is proved.
\end{proof}

\subsection{Proof of Theorem~\ref{theorem 3.3}}
With Lemma~\ref{lemma 3.1}-Lemma~\ref{lemma 3.5}, we begin to prove Theorem~\ref{theorem 3.3}. The idea of this proof is similar to those in \cite{Householder, LHC}.

\begin{proof}
The proof of Theorem~\ref{theorem 3.3} is divided into two parts, orthogonality and residual. 

\subsubsection{The upper bound of orthogonality}
In the beginning, we consider the orthogonality. We primarily focus on $\kappa_{2}(W)$. Since $\kappa_{2}(W)=\frac{\norm{W}_{2}}{\sigma_{n}}$, we estimate $\norm{W}_{2}$ and $\sigma_{n}(W)$ separately. With \eqref{eq:37}, we can have
\begin{equation}
\norm{W}_{2} \le \norm{XY^{-1}}_{2}+\norm{E_{X}Y^{-1}}_{2}. \label{eq:w2}
\end{equation}
According to \eqref{eq:es}, it is clear to see that
\begin{equation}
AY^{-1}=\Omega_{2}\Omega_{1}XY^{-1}+E_{S}Y^{-1}. \label{eq:ay-1e}
\end{equation}
Based on \eqref{eq:ay-1e}, we can get
\begin{equation} \label{eq:nooxy-1}
\begin{split}
\norm{\Omega_{2}\Omega_{1}XY^{-1}}_{2} &\le \norm{E_{S}Y^{-1}}_{2}+\norm{AY^{-1}}_{2} \\ &\le \norm{E_{S}}_{2}\norm{Y^{-1}}_{2}+\norm{AY^{-1}}_{2}. 
\end{split}
\end{equation}
We put \eqref{eq:esk}, \eqref{eq:322} and \eqref{eq:323} into \eqref{eq:nooxy-1} and based on \eqref{eq:rw}, we can have
\begin{equation} \label{eq:ooxy-12}
\begin{split}
\norm{\Omega_{2}\Omega_{1}XY^{-1}}_{2} &\le \norm{E_{S}}_{2}\norm{Y^{-1}}_{2}+\norm{AY^{-1}}_{2} \\ &\le \frac{1.12}{\sqrt{1-\epsilon} \cdot \sigma_{n}(X)} \cdot k\norm{X}_{2}+1.01 \\ &= \frac{1.12}{\sqrt{1-\epsilon}} \cdot k\kappa_{2}(X)+1.01 \\ &\le 1.13, 
\end{split}
\end{equation}
with probability at least $1-p$. With \eqref{eq:2e} and \eqref{eq:ooxy-12}, we can bound $\norm{XY^{-1}}_{2}$ as
\begin{equation}
\norm{XY^{-1}}_{2} \le \frac{1.13}{\sqrt{1-\epsilon}}, \label{eq:xy-1}
\end{equation}
with probability at least $1-p$. With \eqref{eq:rw}, \eqref{eq:32}, \eqref{eq:322} and \eqref{eq:327}, we can bound $\norm{E_{X}Y^{-1}}_{2}$ as
\begin{equation} \label{eq:exy-12}
\begin{split}
\norm{E_{X}Y^{-1}}_{2} &\le \norm{E_{X}}_{2}\norm{Y^{-1}}_{2} \\ &\le 1.05n\sqrt{n}\uu \cdot \norm{W}_{2} \cdot (\sqrt{1+\epsilon}+1.02m\uu \cdot \norm{\Omega}_{F}) \cdot \sqrt{vt_{1}c_{1}^{2}+nt_{2}c_{2}^{2}} \cdot \frac{1.12}{\sqrt{1-\epsilon} \cdot \sigma_{n}(X)} \\ &= 1.05n\sqrt{n}\uu \cdot \norm{W}_{2} \cdot (\sqrt{1+\epsilon}+1.02m\uu \cdot \norm{\Omega}_{F}) \cdot j\norm{X}_{2} \cdot \frac{1.12}{\sqrt{1-\epsilon} \cdot \sigma_{n}(X)} \\ &\le \frac{1.18}{\sqrt{1-\epsilon}} \cdot n\sqrt{n}\uu \cdot \norm{W}_{2} \cdot (\sqrt{1+\epsilon}+1.02m\uu \cdot \norm{\Omega}_{F}) \cdot j\kappa_{2}(X) \\ &\le 0.02\norm{W}_{2}, 
\end{split}
\end{equation}
with probability at least $1-p$. We put \eqref{eq:xy-1} and \eqref{eq:exy-12} into \eqref{eq:w2} and we can have
\begin{equation}
\norm{W}_{2} \le 0.02\norm{W}_{2}+\frac{1.13}{\sqrt{1-\epsilon}}, \label{eq:w2i}
\end{equation}
with probability at least $1-p$. Therefore, we can bound $\norm{W}_{2}$ with \eqref{eq:w2i} as
\begin{equation}
\norm{W}_{2} \le \frac{1.16}{\sqrt{1-\epsilon}}, \label{eq:w2f}
\end{equation}
with probability at least $1-p$. Regarding $\sigma_{n}(W)$, we combine Lemma~\ref{lemma 2.1} with \eqref{eq:37} and we can get
\begin{equation}
\sigma_{n}(W) \ge \sigma_{n}(XY^{-1})-\norm{E_{X}Y^{-1}}_{2}. \label{eq:sigmanw}
\end{equation}
For $\sigma_{n}(XY^{-1})$, with \eqref{eq:rw}, \eqref{eq:esk}, \eqref{eq:322} and \eqref{eq:nay-1} , we can get
\begin{equation} \label{eq:n21xy-1}
\begin{split}
\sigma_{n}(\Omega_{2}\Omega_{1}XY^{-1}) &\ge \sigma_{n}(AY^{-1})-\norm{E_{S}Y^{-1}}_{2} \\ &\ge \sigma_{n}(AY^{-1})-\norm{E_{S}}_{2}\norm{Y^{-1}}_{2} \\ &\ge 0.99-k\norm{X}_{2} \cdot \frac{1.12}{\sqrt{1-\epsilon} \cdot \sigma_{n}(X)} \\ &\ge 0.88, 
\end{split}
\end{equation}
with probability at least $1-p$. With \eqref{eq:2g} and \eqref{eq:n21xy-1}, we can bound $\sigma_{n}(XY^{-1})$ as 
\begin{equation}
\sigma_{n}(XY^{-1}) \ge \frac{0.88}{\sqrt{1+\epsilon}}, \label{eq:nxy-1}
\end{equation}
with probability at least $1-p$. Therefore, we put \eqref{eq:exy-12} and \eqref{eq:nxy-1} into \eqref{eq:sigmanw} and we can bound $\sigma_{n}(W)$ as
\begin{equation} \label{eq:sigmanwf}
\begin{split}
\sigma_{n}(W) &\ge \sigma_{n}(XY^{-1})-\norm{E_{X}Y^{-1}}_{2} \\ &\ge \frac{0.88}{\sqrt{1+\epsilon}}-0.02.\frac{1.16}{\sqrt{1-\epsilon}},
\end{split}
\end{equation}
with probability at least $1-p$. With \eqref{eq:w2f} and \eqref{eq:sigmanwf}, we can bound $\kappa_{2}(W)$ as
\begin{equation} \label{eq:k2wf}
\begin{split}
\kappa_{2}(W) &\le \frac{\norm{W}_{2}}{\sigma_{n}(W)} \\ &\le \frac{\frac{1.16}{\sqrt{1-\epsilon}}}{\frac{0.88}{\sqrt{1+\epsilon}}-0.02 \cdot \frac{1.16}{\sqrt{1-\epsilon}}} \\ &\le \frac{1.16}{0.88 \cdot \sqrt{\frac{1-\epsilon}{1+\epsilon}}-0.03} \\ &= a,
\end{split}
\end{equation}
with probability at least $1-p$. Here, $a=\frac{1.16}{0.88 \cdot \sqrt{\frac{1-\epsilon}{1+\epsilon}}-0.03}$. When \eqref{eq:k2wf} satisfies Assumption~\ref{assumption:2}, $a \le \frac{1}{8\sqrt{mn\uu+n(n+1)\uu}}$. With the theoretical results in \cite{error}, we can bound the orthogonality of RCholeskyQR2 for sparse matrices, $\norm{Q^{\top}Q-I_{n}}_{F}$ as
\begin{equation} 
\begin{split}
\norm{Q^{\top}Q-I_{n}}_{F} &\le \frac{5}{64} \cdot 64(\kappa_{2}(W))^{2} \cdot (mn\uu+n(n+1)\uu) \nonumber \\ &\le 5a^{2} \cdot (mn\uu+n(n+1)\uu), \nonumber
\end{split}
\end{equation}
with probability at least $1-p$. \eqref{eq:372} is proved.

\subsubsection{The upper bound of residual}
Regarding the residual, according to \eqref{eq:310} and \eqref{eq:311}, we can have
\begin{equation}
\begin{split}
QR-X &= (W+E_{3})Z^{-1}(ZY-E_{4})-X \nonumber \\ &= (W+E_{3})Y-(W+E_{3})Z^{-1}E_{4}-X \nonumber \\ &= WY-X+E_{3}Y-QE_{4} \\ &= E_{X}+E_{3}Y-QE_{4}. \nonumber
\end{split}
\end{equation}
Therefore, it is easy to have
\begin{equation}
\norm{QR-X}_{F} \le \norm{E_{X}}_{F}+\norm{E_{3}}_{F}\norm{Y}_{2}+\norm{Q}_{2}\norm{E_{4}}_{F}. \label{eq:344}
\end{equation}
With \eqref{eq:327} and \eqref{eq:w2f}, we can bound $\norm{E_{X}}_{F}$ as
\begin{equation} \label{eq:exff}
\begin{split}
\norm{E_{X}}_{F} &\le 1.05n\sqrt{n}\uu \cdot \norm{W}_{2} \cdot (\sqrt{1+\epsilon}+1.02m\uu \cdot \norm{\Omega}_{F}) \cdot \sqrt{vt_{1}c_{1}^{2}+nt_{2}c_{2}^{2}} \\ &\le 1.05n\sqrt{n}\uu \cdot (\sqrt{1+\epsilon}+1.02m\uu \cdot \norm{\Omega}_{F}) \cdot j\norm{X}_{2} \cdot \frac{1.16}{\sqrt{1-\epsilon}} \\ &\le \frac{1.22}{\sqrt{1-\epsilon}} \cdot n\sqrt{n}\uu \cdot (\sqrt{1+\epsilon}+1.02m\uu \cdot \norm{\Omega}_{F}) \cdot j\norm{X}_{2}, 
\end{split}
\end{equation}
with probability at least $1-p$. With \eqref{eq:222}, \eqref{eq:es} and \eqref{eq:esk}, we can have 
\begin{equation} \label{eq:a22a}
\begin{split}
\norm{A}_{2} &\le \norm{\Omega_{2}\Omega_{1}X}_{2}+\norm{E_{S}}_{2} \\ &\le \sqrt{1+\epsilon} \cdot \norm{X}_{2}+k\norm{X}_{2} \\ &\le (\sqrt{1+\epsilon}+k) \cdot \norm{X}_{2},
\end{split}
\end{equation}
with probability at least $1-p$. Based on the corresponding theoretical results in \cite{error}, together with \eqref{eq:35}, \eqref{eq:36}, \eqref{eq:31}, \eqref{eq:32} and \eqref{eq:a22a}, we can bound $\norm{Y}_{2}$ as 
\begin{equation} \label{eq:y22a}
\begin{split}
\norm{Y}_{2} &\le 1.02\norm{A}_{2} \\ &\le 1.02 \cdot (\sqrt{1+\epsilon}+k) \cdot \norm{X}_{2},
\end{split}
\end{equation}
with probability at least $1-p$. With \eqref{eq:372}, we can bound $\norm{Q}_{2}$ as 
\begin{equation}
\norm{Q}_{2} \le \sqrt{1+d}, \label{eq:q2f}
\end{equation}
with probability at least $1-p$. Here, we set $d=5a^{2} \cdot (mn\uu+n(n+1)\uu)$. Similar to \eqref{eq:y22a}, together with \eqref{eq:31}, \eqref{eq:32} and \eqref{eq:w2f}, we can bound $\norm{Z}_{2}$ as
\begin{equation} \label{eq:z2f}
\begin{split}
\norm{Z}_{2} &\le 1.02\norm{W}_{2} \\ &\le 1.02 \cdot \frac{1.16}{\sqrt{1-\epsilon}} \\ &\le \frac{1.19}{\sqrt{1-\epsilon}},
\end{split}
\end{equation}
with probability at least $1-p$. With Lemma~\ref{lemma 2.2}, \eqref{eq:310}, \eqref{eq:q2f} and \eqref{eq:z2f}, we can bound $\norm{E_{3}}_{F}$ as
\begin{equation} \label{eq:e3f}
\begin{split}
\norm{E_{3}}_{F} &\le 1.02n\uu \cdot \norm{Q}_{F}\norm{Z}_{F} \\ &\le 1.02n^{2}\uu \cdot \norm{Q}_{2}\norm{Z}_{2} \\ &\le 1.02n^{2}\uu \cdot \sqrt{1+d} \cdot \frac{1.19}{\sqrt{1-\epsilon}} \\ &\le \frac{1.22}{\sqrt{1-\epsilon}} \cdot n^{2}\uu \cdot \sqrt{1+d}, 
\end{split}
\end{equation}
with probability at least $1-p$. With Lemma~\ref{lemma 2.2}, \eqref{eq:311}, \eqref{eq:nyf} and \eqref{eq:z2f}, we can bound $\norm{E_{4}}_{F}$ as
\begin{equation} \label{eq:e4f}
\begin{split}
\norm{E_{4}}_{F} &\le 1.02n\uu \cdot \norm{Y}_{F}\norm{Z}_{F} \\ &\le 1.02n\sqrt{n}\uu \cdot \norm{Y}_{F}\norm{Z}_{2} \\ &\le 1.02n\sqrt{n}\uu \cdot 1.02 \cdot (\sqrt{1+\epsilon}+1.02m\uu \cdot \norm{\Omega}_{F}) \cdot \sqrt{vt_{1}c_{1}^{2}+nt_{2}c_{2}^{2}} \cdot  \frac{1.19}{\sqrt{1-\epsilon}} \\ &\le \frac{1.24}{\sqrt{1-\epsilon}} \cdot n\sqrt{n}\uu \cdot (\sqrt{1+\epsilon}+1.02m\uu \cdot \norm{\Omega}_{F}) \cdot j\norm{X}_{2},
\end{split}
\end{equation}
with probability at least $1-p$. Therefore, we put \eqref{eq:exff}, \eqref{eq:y22a}, \eqref{eq:q2f}, \eqref{eq:e3f} and \eqref{eq:e4f} into \eqref{eq:344} and we can bound $\norm{QR-X}_{F}$ as
\begin{equation}
\begin{split}
\norm{QR-X}_{F} &\le \norm{E_{X}}_{F}+\norm{E_{3}}_{F}\norm{Y}_{2}+\norm{Q}_{2}\norm{E_{4}}_{F} \nonumber \\ &\le \frac{1.22}{\sqrt{1-\epsilon}} \cdot n\sqrt{n}\uu \cdot (\sqrt{1+\epsilon}+1.02m\uu \cdot \norm{\Omega}_{F}) \cdot j\norm{X}_{2} \nonumber \\ &+ \frac{1.22}{\sqrt{1-\epsilon}} \cdot n^{2}\uu \cdot \sqrt{1+d} \cdot 1.02 \cdot (\sqrt{1+\epsilon}+k) \cdot \norm{X}_{2} \nonumber \\ &+ \sqrt{1+d} \cdot \frac{1.24}{\sqrt{1-\epsilon}} \cdot n\sqrt{n}\uu \cdot (\sqrt{1+\epsilon}+1.02m\uu \cdot \norm{\Omega}_{F}) \cdot j\norm{X}_{2} \nonumber \\ &\le \frac{(1.22+1.24\sqrt{1+d})}{\sqrt{1-\epsilon}} \cdot (\sqrt{1+\epsilon}+1.02m\uu \cdot \norm{\Omega}_{F}) \cdot jn\sqrt{n}\uu\norm{X}_{2} \\ &+ \frac{1.24\sqrt{1+d}}{\sqrt{1-\epsilon}} \cdot (\sqrt{1+\epsilon}+k) \cdot n^{2}\uu\norm{X}_{2} \\ &= b\norm{X}_{2}, \nonumber
\end{split}
\end{equation}
with probability at least $1-p$. Here, $b=\frac{(1.22+1.24\sqrt{1+d})}{\sqrt{1-\epsilon}} \cdot (\sqrt{1+\epsilon}+1.02m\uu \cdot \norm{\Omega}_{F}) \cdot jn\sqrt{n}\uu+\frac{1.24\sqrt{1+d}}{\sqrt{1-\epsilon}} \cdot (\sqrt{1+\epsilon}+k) \cdot n^{2}\uu$. \eqref{eq:382} is proved.
\end{proof}
\begin{remark}
Compared with Lemma~\ref{lemma 2.9}, \eqref{eq:372} and \eqref{eq:382} in Theorem~\ref{theorem 3.3} show that RCholeskyQR2 has a level of accuracy similar to that of CholeskyQR2 for sparse cases. \eqref{eq:rw} can also be better than the condition of $\kappa_{2}(X)$ for CholeskyQR2. In the proof of Theorem~\ref{theorem 3.3}, one of the most important steps in \eqref{eq:nyf}. The error bound of $\norm{Y}_{F}$ contains the primary information of the sparse $X$ and will influence the whole proof. With the randomized matrix sketching, the way to prove Theorem~\ref{theorem 3.3} is distinguished from that in \cite{CSparse, error}. 
\end{remark}

\subsection{RHC for sparse matrices}
We have done a theoretical analysis of RCholeskyQR2 before. Here, we can form a rounding error analysis of RHC based on Definition~\ref{def:41} for sparse matrices with the same idea

\subsubsection{Settings of RHC}
Similar to Assumption~\ref{assumption:2}, we make an assumption for RHC.

\begin{assumption}[The assumption for matrix sketching of RHC]
\label{assumption:3}
If $\Omega\in \mathbb{R}^{s\times m}$ is a $(\epsilon,p,n)$ oblivious $l_{2}$-subspace embedding in $\mathbb{R}^{m}$, we let
\begin{equation}
\sqrt{\frac{1-\epsilon}{1+\epsilon}} \ge \frac{128}{11}\sqrt{mn\uu+n(n+1)\uu}+\frac{2}{11}. \nonumber
\end{equation}
\end{assumption}

Also, when $X \in \mathbb{R}^{m\times n}$ is a sparse matrix, we let
\begin{equation} \label{eq:rw1}
\kappa_{2}(X) \le \delta_{2}. 
\end{equation}
Here, we have $\delta_{2}=\min(\frac{\sqrt{1-\epsilon}}{10.2m\uu \cdot j\norm{\Omega}_{F}+10t}, \frac{\sqrt{1-\epsilon}}{11.5m\sqrt{n}\uu \cdot (\sqrt{1+\epsilon} \cdot j+1.02m\uu \cdot j\norm{\Omega}_{F}+t)})$. $j$ is defined in \eqref{eq:rw} and $t=1.02sn\uu \cdot (\sqrt{1+\epsilon} \cdot j+1.02m\uu \cdot j\norm{\Omega}_{F})$.

\subsubsection{Theoretical results of RHC for sparse matrices}
With the settings above, we provide rounding error analysis of RHC for sparse matrices, which corresponds to Theorem~\ref{theorem 3.3}.

\begin{theorem}[Rounding error analysis of RHC for sparse matrices]
\label{theorem 3.4}
When $X \in \mathbb{R}^{m\times n}$ is a sparse matrix with $\mbox{rank}(X)=n$ and $m \ge n$, if \eqref{eq:31}, \eqref{eq:32}, Assumption~\ref{assumption:3} and \eqref{eq:rw1} are satisfied, then for $[Q,R]=\mbox{RHC}(X)$, we have
\begin{align}
\norm{Q^{\top}Q-I_{n}}_{F} &\le g, \nonumber \\
\norm{QR-X}_{F} &\le h\norm{X}_{2}, \nonumber
\end{align}
with probability at least $1-p$. Here, we define $g=5c^{2} \cdot (mn\uu+n(n+1)\uu)$, $c=\frac{16}{11 \cdot \sqrt{\frac{1-\epsilon}{1+\epsilon}}-2}$ and $h=\frac{(1.31+1.34\sqrt{1+g})}{\sqrt{1-\epsilon}} \cdot (\sqrt{1+\epsilon} \cdot j+1.02m\uu \cdot j\norm{\Omega}_{F}+t) \cdot n\sqrt{n}\uu+ \frac{1.34\sqrt{1+g}}{\sqrt{1-\epsilon}} \cdot (\sqrt{1+\epsilon}+1.02m\uu \cdot j\norm{\Omega}_{F}+t) \cdot n^{2}\uu$. $j$ and $t$ are defined in \eqref{eq:rw} and \eqref{eq:rw1}. 
\end{theorem}

The proof of Theorem~\ref{theorem 3.4} is similar to the proof of Lemma~\ref{lemma 2.10} in \cite{Householder} and Theorem~\ref{theorem 3.3}. This can be left to readers for practice.

\section{Some extensions, discussions and comparisons of the theoretical results}
\label{sec:extensions}
In this part, we discuss the comparisons between our theoretical results in this work and some existing results. Some extensions are also provided on the basis of the idea in this work.

\subsection{Comparison between RCholeskyQR2 and CholeskyQR2 for sparse matrices}
In this work, we propose RCholeskyQR2 with a step of matrix sketching and do rounding error analysis for the sparse cases. In fact, similar to the proof of Theorem~\ref{theorem 3.3}, we can have some theoretical results of the deterministic CholeskyQR2 for sparse matrices.

\begin{corollary}[Rounding error analysis of CholeskyQR2 for sparse matrices]
\label{cor:c1}
If $X \in \mathbb{R}^{m\times n}$ is a sparse matrix with $\kappa_{2}(X) \le \delta_{3}$, when \eqref{eq:31} and \eqref{eq:32} are satisfied, we have
\begin{align}
\norm{Q^{\top}Q-I_{n}}_{F} &\le 6(mn\uu+n(n+1)\uu), \nonumber \\
\norm{QR-X}_{F} &\le f\norm{X}_{2}. \nonumber
\end{align}
Here, $\delta_{3}=\frac{1}{8j\sqrt{m\uu+(n+1)\uu}}$ and $f=(2.30j+1.21\sqrt{n}) \cdot n\sqrt{n}\uu$. $j$ is defined in \eqref{eq:rw}.
\end{corollary}

Based on Theorem~\ref{theorem 3.3} and Corollary~\ref{cor:c1}, we provide some comparisons between the theoretical results below in Table~\ref{tab:comparison1}-Table~\ref{tab:comparison2}. Table~\ref{tab:comparison3} shows the computational complexity of the starting steps of the two algorithms, matrix sketching+calculating the gram matrix for RCholeskyQR2 and a single step of calculating the gram matrix for CholeskyQR2, which is the primary difference between these two algorithms. Table~\ref{tab:comparison4} and Table~\ref{tab:comparison5} show the comparison between the requirements of $\kappa_{2}(X)$ of CholeskyQR2 and RHC for different types of matrices.

\begin{table}
\caption{Comparison of the upper bounds of $\kappa_{2}(X)$ of the algorithms}
\centering
\begin{tabular}{||c c||}
\hline
$\mbox{Algorithms}$ & $\mbox{Sparse matrices}$ \\
\hline
$\mbox{RCholeskyQR2}$ & $\delta_{1}$ \\
\hline
$\mbox{CholeskyQR2}$ & $\delta_{3}$ \\
\hline
\end{tabular}
\label{tab:comparison1}
\end{table}

\begin{table}
\caption{Comparison of the error bounds of the algorithms for sparse matrices}
\centering
\begin{tabular}{||c c c||}
\hline
$\mbox{Algorithms}$ & $\norm{Q^{\top}Q-I_{n}}_{F}$ & $\norm{QR-X}_{F}$ \\
\hline
$\mbox{RCholeskyQR2}$ & $d$ & $b\norm{X}_{2}$ \\
\hline
$\mbox{CholeskyQR2}$ & $6(mn\uu+n(n+1)\uu)$ & $f\norm{X}_{2}$ \\
\hline
\end{tabular}
\label{tab:comparison2}
\end{table}

\begin{table}
\caption{The computational complexity of the starting steps for the sparse $X$}
\centering
\begin{tabular}{||c c||}
\hline
$\mbox{Algorithms}$ & $\mbox{Computational complexity}$ \\
\hline
$\mbox{RCholeskyQR2}$ & $smn+sn^{2}$ \\
\hline
$\mbox{CholeskyQR2}$ & $\mbox{(nnze(X))}^{2}$ \\
\hline
\end{tabular}
\label{tab:comparison3}
\end{table}

\begin{table}
\caption{Comparison of the upper bounds of $\kappa_{2}(X)$ of CholeskyQR2 for different types of $X$}
\centering
\begin{tabular}{||c c||}
\hline
$\mbox{Types}$ & $\kappa_{2}(X)$ \\
\hline
$\mbox{Sparse}$ & $\delta_{0}$ \\
\hline
$\mbox{Dense}$ & $\delta_{3}$ \\
\hline
\end{tabular}
\label{tab:comparison4}
\end{table}

\begin{table}
\caption{Comparison of the upper bounds of $\kappa_{2}(X)$ of RHC for different types of $X$}
\centering
\begin{tabular}{||c c||}
\hline
$\mbox{Types}$ & $\kappa_{2}(X)$ \\
\hline
$\mbox{Sparse}$ & $\delta$ \\
\hline
$\mbox{Dense}$ & $\delta_{2}$ \\
\hline
\end{tabular}
\label{tab:comparison5}
\end{table}

There is something interesting in Table~\ref{tab:comparison1}. We find that when we take $s$ relatively close to $n$, \textit{e.g.}, $n+1$ or $2n$, for $X \in \mathbb{R}^{m\times n}$ with $m \ge n$, we have $\delta_{1}>\delta_{3}$. Such a comparison indicates that RCholeskyQR2 has better applicability for ill-conditioned sparse matrices than CholeskyQR2, which is an advantage of matrix sketching which has not occurred in other problems before as far as we know. The numerical experiments in Section~\ref{sec:numerical} confirm our finding. If we combine matrix sketching with SCholeskyQR3, a probabilistic $s$ can be taken as $s=11(s\uu+(n+1)\uu) \cdot k_{1}c^{2}$ for SCholeskyQR3 based on the result in \cite{CSparse} for the sparse cases. Table~\ref{tab:comparison2} shows that RCholeskyQR2 and CholeskyQR2 have the same level of accuracy in both orthogonality and residual, as one of the most important results of this work.

Table~\ref{tab:comparison3} shows that when the input $X \in \mathbb{R}^{m\times n}$ is sparse, the comparison of the computational complexity between RCholeskyQR2 and CholeskyQR2 depends on several factors, including $s$ that satisfies $m$, $n$, $n \le s \le m$ and $\mbox{nnze(X)}$. This will be reflected in the experiments of Section~\ref{sec:numerical}. The Gaussian sketch in RCholeskyQR2 may break the sparsity of the original $X$. If we take a different type of sketch, such as multi-sketching with the CountSketch+the Gaussian sketch \cite{Householder, LHC} for $X \in \mathbb{R}^{m\times n}$ with $m=\mathcal{O}(n^{2})$, RCholeskyQR2 can be more efficient. 

Moreover, when comparing Theorem~\ref{theorem 3.4} and Corollary~\ref{cor:c1} with Lemma~\ref{lemma 2.9} and Lemma~\ref{lemma 2.10}, we can find that for the same algorithm, the sparse and common input $X$ will have different results of both $\kappa_{2}(X)$ and accuracy. This can be reflected in some numerical experiments in Section~\ref{sec:numerical}.

\subsection{Comparison between the applicability for different types of matrices}
In this work, we focus on two typical CholeskyQR-type algorithms, CholeskyQR2 and RHC. Theoretical results of them for the sparse cases are shown in Theorem~\ref{theorem 3.4} and Corollary~\ref{cor:c1}. Comparisons between the bounds of $\kappa_{2}(X)$ for different types of $X$ are shown in Table~\ref{tab:comparison4} and Table~\ref{tab:comparison5}. 

These two tables show that the better conditions of $\kappa_{2}(X)$ hold when we use CholeskyQR-type algorithms for sparse matrices. The existing algorithms can deal with more ill-conditioned sparse cases than dense cases. Some experiments are taken in Section~\ref{sec:numerical} to verify the theoretical results.

\subsection{Comparison between different models of sparsity}
We have proposed a model of sparse matrices in \cite{CSparse} based on $\mbox{nnze(X)}$ and the element with the highest absolute value of the input $X$. In this work, a new model based on the existence of dense columns and different types of elements is built in Definition~\ref{def:41}. 

Generally speaking, they are two different ways to bound $\norm{X}_{F}$ based on its structure. Recall the model we used in \cite{CSparse}, we let $k_{1}$ be the number of the non-zero elements of the sparse $X \in \mathbb{R}^{m\times n}$ with $n \le k_{1} \le mn$ when $\mbox{rank}(X)=n$. We also take $c$ as the element with the largest absolute value of $X$. Compared with the model in Definition~\ref{def:41}, we can have $c=\max(c_{1}, c_{2})$. In fact, for SCholeskyQR3, we need to consider the CPU time of taking the shifted item $s$. Although we can take $s=11(m\uu+(n+1)\uu) \cdot (vt_{1}c_{1}^{2}+nt_{2}c_{2}^{2})$ for SCholeskyQR3 theoretically, it is very expensive and time-consuming compared to the choice of $s=11(m\uu+(n+1)\uu) \cdot k_{1}c^{2}$. However, for this work focusing on analysis, the model in Definition~\ref{def:41} can describe the properties of $X$ better when there are a few very large elements in $X$, avoiding potential overestimation due to the absolute values to some extent. We can take different models according to different cases in real problems.

\section{Numerical experiments}
\label{sec:numerical}
In this section, we conduct numerical experiments regarding RCholeskyQR2 and RHC using MATLAB R2023A on our laptop. We show the specifications of our computer below in Table~\ref{tab:C}. We perform numerical experiments for the comparison between RCholeskyQR2 and CholeskyQR2 from  different perspectives, together with the performance of the existing CholeskyQR2 and RHC for different types of $X \in \mathbb{R}^{m\times n}$. We primarily consider the applicability, accuracy, efficiency, and robustness of the algorithms. For the numerical results with respect to the randomized algorithms, we perform the experiments for several times and take the average value. 

\begin{table}
\begin{center}
\caption{The specifications of our computer}
\centering
\begin{tabular}{c|c}
\hline
Item & Specification\\
\hline \hline
System & Windows 11 family(10.0, Version 22000) \\
BIOS & GBCN17WW \\
CPU & Intel(R) Core(TM) i5-10500H CPU @ 2.50GHz  -2.5 GHz \\
Number of CPUs / node & 12 \\
Memory size / node & 8 GB \\
Direct Version & DirectX 12 \\
\hline
\end{tabular}
\label{tab:C}
\end{center}
\end{table}

\subsection{Comparison between RCholeskyQR2 and CholeskyQR2 for sparse matrices}
In this part, we do numerical experiments regarding RCholeskyQR2 and CholeskyQR2 for sparse matrices.

\subsubsection{The applicability and accuracy of the algorithms}
We start with the applicability and accuracy of the algorithms. In this group of numerical experiments, we use the arrowhead matrix which is widely used in many areas in both academia and industry, including graph theory, control theory, and certain eigenvalue problems \cite{Constructing, Li, Eigen}. Unlike many sparse matrices, the arrowhead matrix has dense columns, which means that $v>0$ in Definition~\ref{def:41}. In this part, we define $e_{1ns}=(1,0,0, \cdots, 0,0)^{\top} \in \mathbb{R}^{20}$, $e_{1zs}=(0,1,1, \cdots, 1,1)^{\top} \in \mathbb{R}^{20}$, together with a diagonal matrix $E={\rm diag}(1, \alpha^{\frac{1}{19}}, \cdots, \alpha^{\frac{18}{19}}, \alpha) \in \mathbb{R}^{20\times 20}$. Here, $\sigma$ is a small positive constant with $0<\sigma<1$. Therefore, a small matrix $B \in \mathbb{R}^{20\times 20}$ is built as 
\begin{equation}
B=-5e_{1ns} \cdot e_{1zs}^{\top}-10e_{1zs} \cdot e_{1ns}^{\top}+E. \nonumber
\end{equation}
We build $X \in \mathbb{R}^{20000\times 20}$ with $1000$ $B$ as 
\begin{equation}
X=
\begin{pmatrix}
B \\
B \\
\vdots \\
B \nonumber
\end{pmatrix}.
\end{equation} 
We use $X$ in the block version because it is easy for us to control $\kappa_{2}(X)$ through $\kappa_{B}$

In the following, we do numerical experiments of RCholeskyQR2 and CholeskyQR2 for such an $X$. For RCholeskyQR2, we take $\epsilon=0.5$ and $p=0.6$. Therefore, Assumption~\ref{assumption:2} is satisfied. For RCholeskyQR2, we take $s=200$. We vary $\alpha$ from $10^{-1}$, $10^{-2}$, $10^{-4}$, $10^{-5}$ to $8 \cdot 10^{-2}$ to adjust $\kappa_{2}(X)$. The numerical results on applicability and accuracy are shown in Table~\ref{tab:O} and Table~\ref{tab:R}.

\begin{table}
\caption{Comparison of orthogonality between RCholeskyQR2 and CholeskyQR2 for the sparse $X$}
\centering
\begin{tabular}{||c c c c c c||}
\hline
$\kappa_{2}(X)$ & $419.92$ & $3.99e+03$ & $3.51e+05$ & $3.07e+07$ & $1.30e+09$ \\
\hline
RCholeskyQR2 & $7.67e-15$ & $7.31e-15$ & $9.56e-15$ & $8.38e-15$ & $8.49e-15$ \\
\hline
CholeskyQR2 & $2.51e-15$ & $3.86e-15$ & $7.22e-15$ & $5.71e-15$ & $-$ \\
\hline
\end{tabular}
\label{tab:O}
\end{table}

\begin{table}
\caption{Comparison of orthogonality between RCholeskyQR2 and CholeskyQR2 for the sparse $X$}
\centering
\begin{tabular}{||c c c c c c||}
\hline
$\kappa_{2}(X)$ & $419.92$ & $3.99e+03$ & $3.51e+05$ & $3.07e+07$ & $1.30e+09$ \\
\hline
RCholeskyQR2 & $1.71e-13$ & $5.24e-13$ & $2.82e-13$ & $3.08e-13$ & $2.69e-13$ \\
\hline
CholeskyQR2 & $2.81e-13$ & $3.10e-13$ & $3.11e-13$ & $2.91e-13$ & $-$ \\
\hline
\end{tabular}
\label{tab:R}
\end{table}

According to Table~\ref{tab:O} and Table~\ref{tab:R}, we find that the accuracy of RCholeskyQR2 and CholeskyQR2 is in a similar level, including both orthogonality and residual, which corresponds to Table~\ref{tab:comparison2}. RCholeskyQR2 exhibits better applicability than CholeskyQR2 in this example, which is also reflected in Table~\ref{tab:comparison1}. This is one of the new phenomena of this work, showing a special advantage of the randomized technique in CholeskyQR.

\subsubsection{Robustness of the algorithms}
In this group of numerical experiments, we test the robustness of the algorithms for sparse matrices. This is mainly influenced by $\kappa_{2}(X)$. We still use the arrowhead matrix as shown before. For RCholeskyQR2, we let $s=200$. We vary $\alpha$ in order to change $\kappa_{2}(X)$. We test the times of success every $30$ times for RCholeskyQR2 and CholeskyQR2. Numerical experiments are shown in Table~\ref{tab:T}.

\begin{table}
\caption{Comparison of times of success for RCholeskyQR2 and CholeskyQR2}
\centering
\begin{tabular}{||c c c c c c||}
\hline
$\kappa_{2}(X)$ & $419.92$ & $3.99e+03$ & $3.51e+05$ & $3.07e+07$ & $1.30e+09$ \\
\hline
RCholeskyQR2 & $30$ & $30$ & $30$ & $30$ & $12$ \\
\hline
CholeskyQR2 & $30$ & $30$ & $30$ & $30$ & $0$ \\
\hline
\end{tabular}
\label{tab:T}
\end{table}

Table~\ref{tab:T} shows that RCholeskyQR2 can deal with more ill-conditioned cases with a certain probability. This property of the randomized algorithms is shown in this group of experiments for the first time as far as we know and is distinguished from the deterministic cases. This is in accordance with the results of Theorem~\ref{theorem 3.3} and our choice of $p$ in this group of experiments.

\subsubsection{CPU time of the algorithms}
Here, we test CPU time (s) of RCholeskyQR2 and CholeskyQR2 for sparse matrices. We would like to see the potential influence of $m$, $n$, $s$ and $\mbox{nnze(X)}$. Unlike previous numerical experiments, we use $\mbox{Sprand}$ in MATLAB to generate the input matrix $X$. To test the influence of $m$, we fix $n=50$ and $h=0.05$. To test the influence of $n$, we fix $m=2000$ and $h=0.05$. We keep $\kappa_{2}(X)=10^{6}$, $h=\frac{\mbox{nnze(X)}}{mn}=0.05$ and $s=200$. To see the influence of $s$, we fix $m=2000$, $n=50$, $\kappa_{2}(X)=10^{6}$ and $h=0.05$. We vary $s$ from $50$, $100$, $200$, $500$ to $1000$ and perform numerical experiments. With different $\mbox{nnze(X)}$, we keep $m=2000$, $n=50$, $\kappa_{2}(X)=10^{6}$ and $s=200$. The numerical results are shown below in Table~\ref{tab:t1}-Table~\ref{tab:t4}.

\begin{table}
\caption{Comparison of CPU time for all the algorithms with different $m$ for the sparse $X$}
\centering
\begin{tabular}{||c c c c c c||}
\hline
$m$ & $200$ & $500$ & $1000$ & $2000$ & $4000$ \\
\hline
RCholeskyQR2 & $0.001$ & $0.001$ & $0.003$ & $0.005$ & $0.017$ \\
\hline
CholeskyQR2 & $0.003$ & $0.003$ & $0.004$ & $0.006$ & $0.016$ \\
\hline
\end{tabular}
\label{tab:t1}
\end{table}

\begin{table}
\caption{Comparison of CPU time for all the algorithms with different $n$ for the sparse $X$}
\centering
\begin{tabular}{||c c c c c c||}
\hline
$n$ & $50$ & $100$ & $200$ & $500$ & $1000$ \\
\hline
RCholeskyQR2 & $0.004$ & $0.005$ & $0.006$ & $0.011$ & $0.027$ \\
\hline
CholeskyQR2 & $0.001$ & $0.001$ & $0.006$ & $0.038$ & $0.135$ \\
\hline
\end{tabular}
\label{tab:t2}
\end{table}

\begin{table}
\caption{Comparison of CPU time for all the algorithms with different $s$ for the sparse $X$}
\centering
\begin{tabular}{||c c c c c c||}
\hline
$s$ & $50$ & $100$ & $200$ & $500$ & $1000$ \\
\hline
RCholeskyQR2 & $0.002$ & $0.003$ & $0.005$ & $0.014$ & $0.025$ \\
\hline
CholeskyQR2 & $0.006$ & $0.006$ & $0.006$ & $0.006$ & $0.006$ \\
\hline
\end{tabular}
\label{tab:t3}
\end{table}

\begin{table}
\caption{Comparison of CPU time for all the algorithms with different $\mbox{nnze(X)}$ for the sparse $X$}
\centering
\begin{tabular}{||c c c c c c||}
\hline
$\mbox{nnze(X)}$ & $0.01$ & $0.02$ & $0.05$ & $0.1$ & $0.2$ \\
\hline
RCholeskyQR2 & $0.005$ & $0.005$ & $0.005$ & $0.006$ & $0.008$ \\
\hline
CholeskyQR2 & $0.001$ & $0.002$ & $0.006$ & $0.016$ & $0.024$ \\
\hline
\end{tabular}
\label{tab:t4}
\end{table}

According to Table~\ref{tab:t1}-Table~\ref{tab:t4}, we find that the comparison of efficiency between RCholeskyQR2 and CholeskyQR2 is associated with all of the following factors, $m$, $n$, $s$ and $\mbox{nnze(X)}$. It depends on the real properties of the sparse $X \in \mathbb{R}^{m\times n}$ with the comparison in Table~\ref{tab:comparison3} as $smn+sn^{2}$ versus $\mbox{(nnze(X))}^{2}$. It is an interesting phenomenon which has not occurred in the previous works before.

\subsection{The applicability of the existing algorithms for different types of $X$}
In this part, we show some different phenomena of the existing CholeskyQR2 and RHC for the sparse and dense $X$. We focus mainly on the applicability of the algorithms, which is reflected by $\kappa_{2}(X)$ for the input $X$. For CholeskyQR2, we take the same sparse matrix as in the comparison between RCholeskyQR2 and CholeskyQR2. To see the applicability of CholeskyQR2 for dense matrices, we build a comparison group of a dense matrix. We construct a small $X_{cs} \in \mathbb{R}^{20\times 20}$ using Singular Value Decomposition (SVD). We set
\begin{equation}
X_{cs}=U_{cs} D_{cs} V_{cs}^{T}. \nonumber
\end{equation}
Here, $U_{cs} \in \mathbb{R}^{20\times 20}, V_{cs} \in \mathbb{R}^{20\times 20}$ are orthogonal matrices. $D_{cs}={\rm diag}(1, \sigma^{\frac{1}{19}}, \cdots, \sigma^{\frac{18}{19}}, \sigma) \in \mathbb{R}^{20\times 20}$ is a diagonal matrix with all the singular values of $X_{cs}$. We vary $\sigma$ to control $\kappa_{2}(X_{cs})$. Here, $0<\sigma \le 1$ is a constant. Therefore, we have $\sigma_{1}(X_{cs})=\norm{X_{cs}}_{2}=1$ and $\kappa_{2}(X_{cs})=\frac{1}{\sigma}$. In this way, the dense matrix $X_{c} \in \mathbb{R}^{20000\times 20}$ is constructed with $1000$ $X_{cs}$ as
We build $X_{c} \in \mathbb{R}^{20000\times 20}$ with $1000$ $X_{cs}$ as 
\begin{equation}
X_{c}=
\begin{pmatrix}
X_{cs} \\
X_{cs} \\
\vdots \\
X_{cs} \nonumber
\end{pmatrix}.
\end{equation} 

For RHC, we construct a sparse matrix first. We take $m=2000$, $n=50$ and define some vectors: $e_{1ns}=(1,0,0, \cdots, 0,0)^{\top} \in \mathbb{R}^{50}$, $e_{1zs}=(0,1,1, \cdots, 1,1)^{\top} \in \mathbb{R}^{50}$, $e_{1nb}=(1,0,0, \cdots, 0,0)^{\top} \in \mathbb{R}^{2000}$ and $e_{1zs}=(0,1,1, \cdots, 1,1)^{\top} \in \mathbb{R}^{2000}$, together with a diagonal matrix $E={\rm diag}(1, \theta^{\frac{1}{49}}, \cdots, \theta^{\frac{48}{49}}, \theta) \in \mathbb{R}^{50\times 50}$. Here, $\theta$ is a positive constant with $0<\theta<1$. Moreover, a large matrix $\mathbb{O}_{1950\times 50}$ is formed with all the elements $0$. Therefore, a matrix $P_{sparse} \in \mathbb{R}^{2000\times 50}$ is formed as
\begin{equation}
P_{sparse}=
\begin{pmatrix}
E \\
\mathbb{O}_{1950\times 50} \nonumber
\end{pmatrix}.
\end{equation} 
We build $X \in \mathbb{R}^{2000\times 50}$ as
\begin{equation}
X=-5e_{1nb} \cdot e_{1zs}^{\top}-10e_{1zs} \cdot e_{1ns}^{\top}+P_{sparse}. \nonumber
\end{equation}
For the dense matrix $X_{c}$ as the comparison group, we take an $X_{r} \in \mathbb{R}^{2000\times 50}$, which is generated as $X_{r}=\mbox{rand}(2000,50)/10^{36}$ in MATLAB. Therefore, dense $X_{c}$ is taken as 
\begin{equation}
X_{c}=X+X_{r}. \nonumber
\end{equation}

In this group of numerical experiments, we make a comparison between the sparse matrix $X$ as defined in the previous part and $X_{c}$ to see the applicability of the algorithms for different types of $X$. For CholeskyQR2, we vary $\alpha$ from $0.1$ to $10^{-8}$. We modify $\sigma$ simultaneously to satisfy $\kappa_{2}(X) \approx \kappa_{2}(X_{c})$. For RHC, we vary $\theta$ from $10^{-4}$ to $10^{-20}$. In this case, we find $\kappa_{2}(X) \approx \kappa_{2}(X_{c})$. For these two algorithms, we fix $s=200$. The applicability of the algorithms for different types of $X$  is shown in Table~\ref{tab:od1}-Table~\ref{tab:rd2}. 

\begin{table}
\caption{Comparison of applicability of CholeskyQR2 for the sparse $X$}
\centering
\begin{tabular}{||c c c c c c||}
\hline
$\kappa_{2}(X)$ & $419.92$ & $3.99e+03$ & $3.51e+05$ & $3.07e+07$ & $5.38e+08$ \\
\hline
Orthogonality & $2.51e-15$ & $3.86e-15$ & $7.22e-15$ & $5.71e-15$ & $7.48e-15$ \\
\hline
Residual & $2.81e-13$ & $3.10e-13$ & $3.11e-13$ & $2.91e-13$ & $2.78e-13$ \\
\hline
\end{tabular}
\label{tab:od1}
\end{table}

\begin{table}
\caption{Comparison of applicability of CholeskyQR2 for the dense $X$}
\centering
\begin{tabular}{||c c c c c c||}
\hline
$\kappa_{2}(X)$ & $419.92$ & $3.99e+03$ & $3.51e+05$ & $3.07e+07$ & $5.38e+08$ \\
\hline
Orthogonality & $2.29e-15$ & $2.38e-15$ & $3.13e-15$ & $2.63e-15$ & $-$ \\
\hline
Residual & $9,64e-15$ & $9.42e-15$ & $8.24e-15$ & $7.96e-15$ & $-$ \\
\hline
\end{tabular}
\label{tab:rd1}
\end{table}

\begin{table}
\caption{Comparison of applicability of RHC for the sparse $X$}
\centering
\begin{tabular}{||c c c c c c||}
\hline
$\kappa_{2}(X)$ & $4.13e+06$ & $3.68e+10$ & $3.21e+14$ & $2.76e+18$ & $8.10e+34$ \\
\hline
Applicability & $\mbox{Yes}$ & $\mbox{Yes}$ & $\mbox{Yes}$ & $\mbox{Yes}$ & $\mbox{Yes}$ \\
\hline
\end{tabular}
\label{tab:od2}
\end{table}

\begin{table}
\caption{Comparison of applicability of RHC for the dense $X$}
\centering
\begin{tabular}{||c c c c c c||}
\hline
$\kappa_{2}(X)$ & $4.13e+06$ & $3.68e+10$ & $3.21e+14$ & $2.76e+18$ & $8.10e+34$ \\
\hline
Applicability & $\mbox{Yes}$ & $\mbox{Yes}$ & $\mbox{Yes}$ & $\mbox{Yes}$ & $\mbox{No}$ \\
\hline
\end{tabular}
\label{tab:rd2}
\end{table}

Table~\ref{tab:od1}-Table~\ref{tab:rd2} aligns with Table~\ref{tab:comparison4} and Table~\ref{tab:comparison5}. CholeskyQR-type algorithms may perform differently in dense and sparse cases when some conditions are satisfied. In many cases, CholeskyQR-type algorithms can deal with more ill-conditioned sparse matrices than the common dense matrices. This is due to the structure of the algorithms, especially the connection between the properties of the input $X$ and the $Y$-factor in the preconditioning step. 

\subsection{Comparison between different models of sparsity}
In this part, we make comparison between different models of sparsity as shown in Definition~\ref{def:41} and \cite{CSparse}. We take CholeskyQR2 as an example. According to Definition~\ref{def:41}, they all satisfy $c_{1} \ge c_{2}$. Here, we make a difference with $c_{1}<c_{2}$. We still focus on the arrowhead matrix for RHC as defined above. We change the diagonal matrix $E \in \mathbb{R}^{50\times 50}$ into $E={\rm diag}(1, 10, 10 \cdot \theta^{\frac{1}{48}}, \cdots, 10 \cdot \theta^{\frac{47}{48}}, 10 \cdot \theta) \in \mathbb{R}^{50\times 50}$. We form $P_{sparse} \in \mathbb{R}^{2000\times 50}$ in the same way. We construct $X \in \mathbb{R}^{m\times n}$ as 
\begin{equation}
X=-5e_{1nb} \cdot e_{1zs}^{\top}-e_{1zs} \cdot e_{1ns}^{\top}+P_{sparse}. \nonumber
\end{equation}

For such a $X$, we have $m=2000$, $n=50$, $c_{1}=1$, $c_{2}=10$, $v=1$, $t_{1}=2000$ and $t_{2}=2$ according to Definition~\ref{def:41}. Following the setting in \cite{CSparse}, we have $m=2000$, $n=50$, $k_{1}=2098$ and $c=10$. In the same way as $\delta_{3}$ in Corollary~\ref{cor:c1}, we can get $\delta_{4}=\frac{1}{8j_{1}\sqrt{m\uu+(n+1)\uu}}$ as the sufficient condition of CholeskyQR2 for sparse matrices under the model of sparsity in \cite{CSparse}. Here, $j_{1}=\frac{\sqrt{k_{1}c^{2}}}{\norm{X}_{2}}$. Here, $k_{1}=98$. For the input $X$, we take $\theta=10^{-6}$ and record $\delta_{3}$ and $\delta_{4}$ in Table~\ref{tab:delta}

\begin{table}
\caption{Comparison of $\delta$ between different models of sparsity}
\centering
\begin{tabular}{||c c||}
\hline
$\kappa_{2}(X)$ & $3.90e+06$ \\
\hline
$\delta_{3}$ & $1.07e+05$ \\
\hline
$\delta_{4}$ & $2.56e+04$ \\
\hline
\end{tabular}
\label{tab:delta}
\end{table}

In Table~\ref{tab:delta}, we find that $\delta_{3}>\delta_{4}$ and is closer to $\kappa_{2}(X)$. After the numerical experiment, we find that CholeskyQR2 can deal with this case. Therefore, we can say that in many cases, our model of sparsity proposed in Definition~\ref{def:41} is better than that of \cite{CSparse} and is closer to the real cases, which depends mainly on the comparison between $vt_{1}c_{1}^{2}+nt_{2}c_{2}^{2}$ and $k_{1}c^{2}$. Readers can conduct similar experiments for RHC.

\section{Conclusions}
\label{sec:conclusions}
This work focuses on rounding error analysis of randomized CholeskyQR for sparse matrices. We provide a new model of the sparse matrices and do rounding error analysis of some typical CholeskyQR-type algorithms with matrix sketching for the sparse cases. We compare the theoretical results with our new model with the cases based on the model in \cite{CSparse}. Numerical experiments demonstrate our findings and show some new phenomena of randomized CholeskyQR-type algorithms for the sparse cases, which are distinguished from the common dense cases. We also test the properties of randomized CholeskyQR-type algorithms, including applicability, accuracy, efficiency, and robustness.

\section*{Acknowledgments}
We express our great attitude to Professor Zhonghua Qiao from the Hong Kong Polytechnic University for his kind support and help. We are also grateful for the ideas provided by Professor Tiexiang Li from Southeast University, and for the discussions with her about sparse matrices. Moreover, we appreciate the discussion with  Professor Valeria Simoncini from University of Bologna and Professor Michael Kwok-Po Ng from Hong Kong Baptist University regarding randomized linear algebra and its connections with CholeskyQR.

\section*{Conflict of interest}
The authors declare that they have no conflict of interest.

\section*{Data availability}
The authors declare that all data supporting the findings of this study are available within this article.

\bibliography{references}

@article{Novel,
  title =	 {{A Novel Randomized XR-Based Preconditioned CholeskyQR Algorithm}},
  author = {Fan, Y. and Guo, Y. and Lin, T.},
  journal = {{arxiv preprint arXiv:2111.11148}},
  year =	 {2021}
}

@article{Randomized,
  title =	 {{Randomized CholeskyQR factorizations}},
  author = {Balabanov, O.},
  journal =	{{arxiv preprint arXiv:2210.09953}},
  year =	 {2022}
}

@book{Higham,
  author =	 {Higham, N.J.},
  title =	 {Accuracy and Stability of Numerical Algorithms},
  publisher =	 {SIAM},
  address =	 {Philadelphia, PA, USA},
  year =	 {2002},
  edition =	 {second~ed.},
}

@article{Shifted,
  title={{Shifted Cholesky QR for computing the QR factorization of ill-conditioned matrices}},
  author={Fukaya, T. and Kannan, R. and Nakatsukasa, Y. and Yamamoto, Y. and Yanagisawa, Y.},
  journal={{SIAM Journal on Scientific Computing}},
  volume={42},
  number={1},
  pages={{A477--A503}},
  year={2020},
  publisher={{SIAM}}
}

@article{error,
  title={{Roundoff error analysis of the CholeskyQR2 algorithm}},
  author={Yamamoto, Y. and Nakatsukasa, Y. and Yanagisawa, Y. and Fukaya, T.},
  journal={{Electronic Transactions on Numerical Analysis}},
  volume={44},
  number={01},
  year={2015}
}

@inproceedings{2014,
  title={{CholeskyQR2: a simple and communication-avoiding algorithm for computing a tall-skinny QR factorization on a large-scale parallel system}},
  author={Fukaya, T. and Nakatsukasa, Y. and Yanagisawa, Y. and Yamamoto, Y.},
  booktitle={2014 5th workshop on latest advances in scalable algorithms for large-scale systems},
  pages={31--38},
  year={2014},
  organization={{IEEE}}
}

@book{Perturbation,
  author =	 {Stewart, G.W. and Sun, J.},
  title =	 {Matrix perturbation theory},
  publisher =	 {Academic Press},
  address =	 {San Diego, CA, USA},
  year =	 {1990},
  edition =	 {sixth~ed.}
}

@article{Fast, 
  author = {Yuster, R. and Zwick, U.},
  title =	{{Fast sparse matrix multiplication}},
  journal = {{ACM Transactions on Algorithms}},
  volume = {1},
  year = {2005},
  pages = {2-13}
}

@article{Constructing, 
  author = {Borobia, A.},
  title =	{{Constructing matrices with prescribed main-diagonal submatrix and characteristic polynomial}},
  journal = {{Linear Algebra and its Applications}},
  volume={418},
  issues={2-3},
  year={2006},
  pages={886-890},
  publisher={{Elsevier}}
}

@article{Eigen, 
  author = {Peng, J. and Hu, X. and Zhang, L.},
  title =	{{Two inverse eigenvalue problems for a special kind of matrices}},
  journal = {{Linear Algebra and its Applications}},
  volume={416},
  issues={2-3},
  year={2006},
  pages={336-347},
  publisher={{Elsevier}}
}

@article{Li,
  title={{The inverse eigenvalue problem for generalized Jacobi matrices with functional relationship}},
  author={Li, Z. and Wang, Y. and Li, S.},
  journal={{2015 12th International Computer Conference on Wavelet Active Media Technology and Information Processing (ICCWAMTIP)}},
  year={2015},
  pages={473-475}
}

@article{Householder,
  title={{Analysis of Randomized Householder-Cholesky QR factorization with multisketching}},
  author={Higgins, A. and Szyld, D. and Boman, E. and Yamazaki, I.},
  journal={{Numerische Mathematik}},
  year={2025},
  volume={157},
  pages={1695-1737}
}

@Book{Introduction,
  author =	 {Stoer, J. and Bulirsch, R.},
  title =	 {Introduction to numerical analysis},
  publisher =	 {Springer},
  address =	 {New York},
  year =	 {2002},
  edition =	 {3rd ed.}
}

@article{Columns,
author = {Fan, Y. and Guan, H. and Qiao, Z.},
year = {2025},
pages = {},
title = {{An Improved Shifted CholeskyQR Based on Columns}},
volume = {104},
number = {2},
journal = {{Journal of Scientific Computing}}
}

@article{CSparse,
  title =	 {{Shifted CholeskyQR for sparse matrices}},
  author = {Guan, H. and Fan, Y.},
  journal = {{arxiv preprint arXiv:2410.06525}}, 
  year =	 {2024}
}

@inproceedings{4031351,
  author={Sarlos, T.},
  booktitle={2006 47th Annual IEEE Symposium on Foundations of Computer Science (FOCS'06)}, 
  title={{Improved Approximation Algorithms for Large Matrices via Random Projections}}, 
  year={2006},
  pages={143-152}
}

@article{rgs,
author = {Balabanov, O. and Grigori, L.},
title = {{Randomized Gram--Schmidt Process with Application to GMRES}},
journal = {{SIAM Journal on Scientific Computing}},
volume = {44},
number = {3},
pages = {{A1450-A1474}},
year = {2022}
}

@inproceedings{pmlr,
  title = 	 {{How to Fake Multiply by a Gaussian Matrix}},
  author = 	 {Kapralov, M. and Potluru, V. and Woodruff, D.},
  booktitle = 	 {Proceedings of The 33rd International Conference on Machine Learning},
  pages = 	 {2101--2110},
  year = 	 {2016},
  editor = 	 {Balcan, Maria Florina and Weinberger, Kilian Q.},
  volume = 	 {48},
  series = 	 {Proceedings of Machine Learning Research},
  address = 	 {{New York, New York, USA}},
  month = 	 {{20--22 Jun}},
  publisher =    {PMLR}
}

@article{pylspack,
author = {Sobczyk, A. and Gallopoulos, E.},
title = {{pylspack: Parallel Algorithms and Data Structures for Sketching, Column Subset Selection, Regression, and Leverage Scores}},
year = {2022},
issue_date = {{December 2022}},
publisher = {{Association for Computing Machinery}},
address = {{New York, NY, USA}},
volume = {48},
number = {4},
journal = {{ACM Transactions on Mathematical Software}},
month = {{December}},
articleno = {44},
numpages = {27}
}

@article{LHC,
  title =	 {{Randomized LU-Householder CholeskyQR}},
  author = {Guan, H. and Fan, Y.},
  journal = {{arxiv preprint arXiv:2412.06551}},
  year =	 {2024}
}

@article{Estimating,
author = {Sobczyk, A. and Gallopoulos, E.},
title = {{Estimating Leverage Scores via Rank Revealing Methods and Randomization}},
journal = {{SIAM Journal on Matrix Analysis and Applications}},
volume = {42},
number = {3},
pages = {1199-1228},
year = {2021},
}

\end{document}